\documentclass[11pt]{amsart}
\usepackage{amsmath}
\usepackage[dvips]{epsfig}
\theoremstyle{plain}
\newtheorem{theorem}{Theorem}[section]
\newtheorem{lemma}[theorem]{Lemma}

\newtheorem{conjecture}[theorem]{Conjecture}

\theoremstyle{definition}
\newtheorem{remark}[theorem]{Remark}

\numberwithin{equation}{section}
\gdef\x{\xi}

\begin{document}

\title[Isometric Embeddings and the Characteristic Variety]
{The Linearized System for Isometric Embeddings\\
and Its Characteristic Variety}
\author[Han]{Qing Han}
\address{Department of Mathematics\\
University of Notre Dame\\
Notre Dame, IN 46556} \email{qhan@nd.edu}
\author[Khuri]{Marcus Khuri}
\address{Department of Mathematics\\
Stony Brook University\\ Stony Brook, NY 11794}
\email{khuri@math.sunysb.edu}
\thanks{The first author acknowledges the support of NSF
Grant DMS-0654261. The second author acknowledges the support of
NSF Grant DMS-1007156 and a Sloan Research Fellowship.}
%\date{\today}
\begin{abstract}
In this paper we prove a conjecture of Bryant, Griffiths, and Yang
concerning the characteristic variety for the determined isometric
embedding system. In particular, we show that the characteristic
variety is not smooth for any dimension greater than 4. This is
accomplished by introducing a smaller yet equivalent linearized
system, in an appropriate way, which facilitates analysis of the
characteristic variety.
\end{abstract}
\maketitle

\section{Introduction}
\label{Sec-Intro}

Let $(M^{n},g)$ be an $n$-dimensional Riemannian manifold.
It is a classical problem to find
an isometric embedding
\begin{equation}\label{0.1}
(M^{n},g)\hookrightarrow\mathbb{R}^{N}.
\end{equation}
The existence of such a global isometric embedding for some $N$
was first proved by Nash \cite{Nash1956}. A better $N$ was later
found by G\"{u}nther \cite{Gunther1989}. In this paper we will
focus exclusively on the local isometric embedding problem.

Suppose that the metric $g=g_{ij}(x)dx^{i}dx^{j}$ is given in a
neighborhood of a point, say $(x^{1},\ldots,x^{n})=0$.  Then we seek
$N$ functions $\{u^{i}\}_{i=1}^{N}$ such that
\begin{equation*}
g=(du^{1})^{2}+\cdots+(du^{N})^{2}.
\end{equation*}
Therefore (\ref{0.1}) is equivalent to the local solvability of the
following first order nonlinear system
\begin{equation}\label{0.2}
\sum_{k=1}^{N}\partial_{x^i} u^{k}\cdot\partial_{x^j}
u^{k}=g_{ij} \quad\text{for }1\le i,j\le n.
\end{equation}
There are $n(n+1)/2$ equations and $N$ unknowns in this system.
Hence, this system is underdetermined if $N>n(n+1)/2$ and
overdetermined if $N<n(n+1)/2$. In the following, we will always
assume that $N=n(n+1)/2$.

For $n=2$, the existence of local isometric embeddings of surfaces
into $\mathbb R^3$ is equivalent to the existence of local
solutions of Darboux's equation, a fully nonlinear equation of the
Monge-Amp\`{e}re class. The type of Darboux's equation is
determined solely by the Gauss curvature. More precisely it is
elliptic if the Gauss curvature is positive, hyperbolic if the
Gauss curvature is negative, and degenerate if the Gauss curvature
has zeroes. Under various assumptions on the Gauss curvature, the
existence of local isometric embeddings was proven by Lin
\cite{Lin1985}, Lin \cite{Lin1986}, Han, Hong and Lin
\cite{HanHongLin2003}, Han \cite{Han2005}, Han \cite{Han200?}, Han
and Khuri \cite{HanKhuri}, Khuri \cite{Khuri1}, and Khuri
\cite{Khuri2}. (See \cite{Han-Hong2006} for details.)

The situation becomes more complicated for $n\ge 3$. Bryant,
Griffiths and Yang \cite{BGY} studied the local isometric
embedding problem for $n$-dimensional Riemannian manifolds and
analyzed the structure of the characteristic variety for the
linearized system. They proved the existence of local isometric
embeddings of 3-dimensional Riemannian manifolds into $\mathbb
R^6$ under an appropriate assumption on the curvature. Later on
Nakamura and Maeda \cite{Naka-Maeda1989} (independently Goodman
and Yang \cite{Goodman-Yang}) proved the existence of local
isometric embeddings of 3-dimensional Riemannian manifolds into
$\mathbb R^6$ when the Riemann curvature tensor does not vanish.

The difficulty in studying isometric embeddings of higher
dimensional Riemannian manifolds lies with the following two
related facts. First the differential system (\ref{0.2}) is very
large, consisting of $n(n+1)/2$ equations for $n(n+1)/2$ unknowns.
Second and most importantly, it is not at all clear how the
curvature determines the type of this system. Hence, a natural
first step is to investigate whether this huge system can be
simplified. Since (\ref{0.2}) is nonlinear this requires an
understanding of the linearized system. However due to its
invariance under the orthogonal group, (\ref{0.2}) is highly
degenerate in that every direction is characteristic, so a direct
study of the linearization appears to be futile. It is thus
necessary to replace the linearized equations by an equivalent
system which is easier to analyze. Bryant, Griffiths and Yang
\cite{BGY} pointed out that the linearization of (\ref{0.2}) is in
fact equivalent to a smaller differential system of $n$ equations
for $n$ unknowns. One may then focus attention on the structure of
the characteristic variety for this new system. For $n=3$, they
proved that the characteristic variety is smooth whenever certain parameters
in the linearized equations lie in appropriate ranges. The smoothness of the
characteristic variety plays an essential role in the existence results in
\cite{BGY}, \cite{Naka-Maeda1989}, and \cite{Goodman-Yang}. For
higher dimensions, they proved that the characteristic variety is
smooth for $n=4$ and not smooth for $n=6, 10, 14,\cdots$. They
also conjectured that the characteristic variety is not smooth for
any $n\ge 5$.

In this paper, we will put this equivalent linearized system in an
explicit form by introducing appropriate parameters. Based on this
explicit expression, we will prove that the characteristic variety
is indeed not smooth for all higher dimensions when these
parameters are sufficiently small.

To motivate our study, let $u$ be a solution of
(\ref{0.2}) and consider the linearization of (\ref{0.2}) at $u$. It
has the following form
\begin{equation}\label{0.3}
\partial_ju\cdot\partial_iv
+\partial_iu\cdot\partial_jv=f_{ij}\quad\text{for any }1\le i,j\le n.
\end{equation}
To find a better equation for $v$, we rewrite this as
\begin{equation}\label{0.4}
\partial_i(\partial_ju\cdot v)
+\partial_j(\partial_iu\cdot v)-2\partial_{ij}u\cdot
v=f_{ij}\quad\text{for any }1\le i,j\le n.
\end{equation}
We note that the inner product $\partial_ju\cdot v$ is a component
of the projection of $v$ into the tangent space spanned by
$\{\partial_1u,\cdots,\partial_nu\}$. It is clear from (\ref{0.4})
that the derivatives are only applied to tangential components of
$v$. This suggests that we should decompose $v$ relative to the
tangent space and normal space of the embedding $u$. In other
words, we uncouple the system by breaking $v$ into tangential and
normal components. It turns out that the normal components of $v$
satisfy an algebraic system which we solve first. Then the
tangential components of $v$ satisfy a differential system of
first order, which consists of $n$ equations for $n$ unknowns.
This new system is much easier to study than (\ref{0.3}).
Moreover, the curvature tensor of $g$ has an explicit expression
in terms of coefficients of this new system. In summary the
linearized isometric embedding system, an $n(n+1)/2\times
n(n+1)/2$ system, can be reduced to an $n\times n$ system which
can be put into an explicit form. We point out that (\ref{0.4})
appears in \cite{BGY} as (2.c.5) and (4.d.4), and that the
equivalent $n\times n$ system is given by (4.d.5).

Now we describe this new $n\times n$ differential system in a more
explicit way. To start with, let $g$ be a smooth metric in a
neighborhood of the origin in $\mathbb R^n$. For $i,j,k=1,\ldots,n$,
let ${\bf c}=\{c_{i}^{kj}\}_{k\neq j}$ be a collection of parameters with
$$c_{i}^{kj}=c_{i}^{jk}\quad\text{for any } i, j, k \text{ with }j\neq k,$$
and set
\begin{align*}
&c_{i}^{ii}=1\quad\text{for any }i,\\
&c_{i}^{jj}=0\quad\text{for any }i\neq j. \end{align*}
There are $n^{2}(n-1)/2$ elements in $\bf c$.
Now define $n\times n$ matrices $A^1, \ldots, A^n$ by
$$(A^k)_{ij}=(c_{i}^{kj}).$$
We may then formulate a differential system in the following way
\begin{equation}\label{0.6}
A^1\partial_{1}V+
\cdots+
A^n\partial_{n}V=F,
\end{equation}
where $V, F$ are vector-valued functions of $n$ components. This
system has constant coefficients which are related to the
curvature of $g$ at the origin; this will be described in detail
in Section \ref{section-reduction}. An important result, stated in
Lemma \ref{lemma-linearization}, asserts that {\it (\ref{0.6}) is
the equivalent linearized isometric embedding system evaluated at
$x=0$}.

As mentioned above, the main step of changing (\ref{0.3}) into an
equivalent $n\times n$ differential system was already observed in
\cite{BGY}, and a version of this equivalent system was given by
(4.d.5). However, the explicit form of (\ref{0.6}) in this paper
is new. As we will see, this explicit form has natural
advantages when it comes to analyzing the characteristic variety
in detail.

To better understand (\ref{0.6}), it is important to study its
characteristic variety. For each $\xi=(\xi_1,\cdots,
\xi_n)\in\mathbb R^n$ define
$$P=P(\xi,{\bf c})=\sum_{i=1}^n\xi_kA^k,$$
where ${\bf c}$ is as above. This is the {\it principal symbol}, and
the associated {\it characteristic variety} is then given by
$$\Sigma({\bf c})=\{\xi\in\mathbb R^n\setminus\{0\}\mid \det P(\xi, {\bf c})=0\}.$$

In dimension 3, and under the assumption that the matrices $A^{k}$ are symmetric,
it was shown (\cite{BGY}) that $\Sigma({\bf c})$ is
smooth in $\mathbb R^3\setminus\{0\}$ except for three choices of
${\bf c}$. Moreover in higher dimensions, it was shown that
$\Sigma({\bf c})$ is generally smooth in $\mathbb R^n\setminus\{0\}$ for
$n=4$ but not smooth in $\mathbb R^n\setminus\{0\}$ for $n=6, 10,
14, \cdots$. (See Corollary (1.c.6) in \cite{BGY}.) The following
conjecture was posed in \cite{BGY}.

\begin{conjecture}
$\Sigma({\bf c})$ is not smooth in $\mathbb R^n\setminus\{0\}$ for $n\ge 5$.
\end{conjecture}

Based on the explicit form of the principal symbol, we will give
an affirmative answer to this conjecture for small ${\bf c}$. We will
say that the parameters ${\bf c}$ satisfy a \textit{generic condition} if
they satisfy a finite number of (homogeneous) polynomial inequalities. For a precise
statement of the following result, see Theorems 4.3 and 4.4 below.

\begin{theorem}\label{theorem-nonempty}
For any $n\ge 5$ and any small ${\bf c}$ satisfying a generic condition,
$\Sigma({\bf c})$ is not smooth in $\mathbb R^n\setminus\{0\}$.
\end{theorem}

As is shown in the proof, the generic condition will be given explicitly. In the case of dimension 4, we will show that under generic
conditions and the smallness assumption the characteristic variety is smooth; a
more general result of this nature has already been obtained in \cite{BGY}. Our proof
here is different.

Let $\Sigma_{\text{sing}}({\bf c})$ be the singular part of
$\Sigma({\bf c})$. We will prove that for $n=5$ the set
$\Sigma_{\text{sing}}({\bf c})\cap \mathbb{P}^4$ generically consists of
exactly $10+\alpha+2\beta+3\gamma$ points when ${\bf c}$ is sufficiently small, where
$\alpha$, $\beta$, $\gamma$ are nonnegative integers with $\alpha+\gamma=10$ and $\beta\leq 5$,
and where $\mathbb{P}^{4}$ denotes real projective space. These points can be located in
terms of the components of ${\bf c}$. In the general case $n\ge 6$, it will be shown
that $\Sigma_{\text{sing}}({\bf c})\cap \mathbb S^{n-1}$ contains a
smooth surface of dimension $n-5$ for sufficiently small ${\bf
c}$ (also assuming generic conditions). We believe that
$\Sigma_{\text{sing}}({\bf c})\cap \mathbb S^{n-1}$ itself consists of
an algebraic variety of dimension $n-5$, possibly under extra assumptions
on ${\bf c}$. Such an algebraic variety may have singularities. For example
for $n=6$, $\Sigma_{\text{sing}}({\bf c})\cap \mathbb S^{5}$ should consist
of finitely many curves which intersect at finitely many points.
We note that it would be desirable to have a stratification of
$\Sigma_{\text{sing}}({\bf c})\cap \mathbb S^{n-1}$ for any $n\ge
5$. Of course, it also remains a challenge to remove the smallness
assumption on $\bf c$.

This paper is organized in the following way. In Section
\ref{section-approximate} we construct appropriate approximate
solutions to the isometric embedding system. In
Section \ref{section-reduction} we discuss the linearized equations
and introduce our explicit equivalent
$n\times n$ system. Lastly in Section \ref{section-HighDim}, we
examine the characteristic variety and prove Theorem \ref{theorem-nonempty}.

\section{Constructing Approximate
Solutions}\label{section-approximate}

In this section we construct an appropriate approximate isometric
embedding, which plays an important role in later discussions.

We first briefly review the theory of surfaces in
Euclidean spaces. In this paper, we will exclusively discuss
$n$-dimensional surfaces in Euclidean space of dimension $s_n$.
Here
$$s_n=\frac12n(n+1).$$ Hence, the codimension is
$$s_n-n=\frac12n(n-1).$$
The
Einstein summation convention will be used with respect to indices
$1\le i,j,k,\cdots\le n$ and $1\leq\mu, \tau,\cdots\leq n(n-1)/2$.

Let $u:\mathbb{R}^{n}\rightarrow\mathbb{R}^{n(n+1)/2}$
be a smooth embedding. Denote the corresponding embedded submanifold by
$\mathcal{M}^{n}$. Then
$\{\partial_i u(x)\}_{i=1}^{n}$ spans $T_{x}\mathcal{M}^{n}$ for
each $x$. Let
$\{N_{\mu}(x)\}_{\mu=1}^{n(n-1)/2}$ span
$(T_{x}\mathcal{M}^{n})^{\bot}$, the orthogonal complement of
$T_{x}\mathcal{M}^{n}$ in $\mathbb{R}^{n(n+1)/2}$.
Denote the induced metric on $\mathcal{M}^{n}$ by
\begin{equation*}
p_{ij}=\partial_i u\cdot\partial_j u.
\end{equation*}
Now recall the fundamental equations for the surface induced by
$u$. Namely $\partial_{ij}u$ has a decomposition into its
tangential and normal components, with respect to $u$, given by
\begin{equation}\label{1.4}
\partial_{ij}u=\Gamma_{ij}^{k}\partial_k u+H_{ij},
\end{equation}
where $\Gamma_{ij}^{k}$ are Christoffel symbols corresponding to
$p_{ij}$ and $H_{ij}$ is the second fundamental form. Moreover we
have
\begin{equation}\label{1.5}
\partial_jN_\mu\cdot\partial_iu=-N_\mu\cdot\partial_{ij}u=-N_{\mu}\cdot H_{ij}.
\end{equation}
By setting $H_{ij}^{\mu}=H_{ij}\cdot N_{\mu}$, $1\leq\mu\leq
n(n-1)/2$, we have
$$H_{ij}=\sum_{\mu=1}^{n(n-1)/2}H_{ij}^\mu N_{\mu}.$$
Also the Gauss equations are given by
\begin{equation}\label{1.13}
R_{ijkl}=\sum_{\mu=1}^{n(n-1)/2}H_{ik}^{\mu}H_{jl}^{\mu}
-H_{il}^{\mu}H_{jk}^{\mu},
\end{equation}
where $R_{ijkl}$ is the curvature tensor associated with the metric $p_{ij}$.

Next we construct approximate solutions to the isometric embedding
system. Let $g$ be a metric defined in a neighborhood of the
origin in $\mathbb R^n$. Take normal coordinates so that
\begin{equation}\label{2.1}
g_{ij}(0)=\delta_{ij}\quad\text{ and }\quad
\partial_{k}g_{ij}(0)=0\quad\text{for any }1\leq i,j,k\leq n.
\end{equation}
Consider a map
$u=(u^{1},\ldots,u^{n(n+1)/2}):\mathbb{R}^{n}\rightarrow
\mathbb{R}^{n(n+1)/2}$ whose components are given by
\begin{align}\label{2.1a}\begin{split}
u^{l}=&x^{l}+\frac{1}{3!}\!\!\sum_{1\leq i,j,k\leq
n}\!\!\!\!\alpha_{ijk}^{l}x^{i}x^{j}x^{k}\quad\text{for any }l=1,\ldots, n,\\
u^{n+\mu}=&\frac{1}{2}\!\!\sum_{1\leq i,j\leq
n}\!\!\!\! h_{ij}^{\mu}x^{i}x^{j}\quad\text{for any }\mu=1,\ldots, n(n-1)/2,
\end{split}\end{align}
for some constants $\alpha_{ijk}^{l}$ and $h_{ij}^\mu$,
$i,j,k,l=1,\ldots,n$ and $\mu=1,\ldots,n(n-1)/2$. We will now
investigate whether the induced metric $du\cdot du$ agrees with
the given metric $g$ up to order two at the origin.

First note that for any $i,j,k=1,\ldots, n$ and $\mu=1,\ldots,
n(n-1)/2$, we have
\begin{align}\label{2.3}\begin{split}
&u(0)=0,\\
&\partial_ju^{i}(0)=\delta_{ij},\quad
\partial_ju^{n+\mu}(0)=0,\\
&\partial_{ij}u^{k}(0)=0,\quad
\partial_{ij}u^{n+\mu}(0)=h_{ij}^{\mu},\\
&\Gamma_{ij}^{k}(0)=0,\\
&N_{\mu}(0)=(0,\ldots,0,\stackrel{n+\mu}{1},0,\ldots,0).
\end{split}\end{align}
Furthermore according to (\ref{2.1}), the metric induced by the embedding $u$
agrees with the given metric $g$ up to order one at the origin.
%\begin{eqnarray*}
%g_{ij}(0)&=&\partial_{i}u(0)\cdot \partial_{j}u(0)\quad\text{for any }1\leq %i,j\leq n,\\
%\partial_{k}g_{ij}(0)&=&\partial_{ik}u(0)\cdot %\partial_{j}u(0)+\partial_{i}u(0)\cdot
%\partial_{jk}u(0)\quad\text{for any } 1\leq i,j,k\leq n.
%\end{eqnarray*}
In order for such a metric to agree with $g$ up to order two at
the origin, we must have
\begin{equation}\label{2.4}
\partial_{kl}g_{ij}(0)=\partial_{ik}u(0)\cdot \partial_{lj}u(0)+\partial_{il}u(0)\cdot \partial_{jk}u(0)
+\partial_{j}u(0)\cdot \partial_{ikl}u(0)+\partial_{i}u(0)\cdot \partial_{ljk}u(0).
\end{equation}
Recall the expression for the curvature tensor in normal
coordinates
\begin{equation}\label{2.5}
R_{ijkl}=\frac{1}{2}(\partial_{il}g_{jk}+\partial_{jk}g_{il}
-\partial_{ik}g_{jl}-\partial_{jl}g_{ik}).
\end{equation}
Hence (\ref{2.4}) implies that
\begin{equation*}
R_{ijkl}(0)=\partial_{ik}u(0)\cdot \partial_{jl}u(0)-\partial_{il}u(0)\cdot \partial_{jk}u(0).
\end{equation*}
Therefore we have
\begin{equation}\label{2.2}
R_{ijkl}(0)=\sum_{\mu=1}^{n(n-1)/2}h_{ik}^{\mu}h_{jl}^{\mu}
-h_{il}^{\mu}h_{jk}^{\mu}.
\end{equation}
These are simply the Gauss equations when $h_{ij}^\mu$ are
interpreted as the coefficients of the second fundamental form at
$x=0$. In other words, the Gauss equations (\ref{2.2}) are a
necessary condition for $u$ to be an approximate solution of the
isometric embedding system up to order two. Next, we prove that it
is also a sufficient condition.

\begin{lemma}\label{lemma2.1}
Let $g$ be a smooth metric defined in a neighborhood of the origin
in $\mathbb R^n$ and let $R_{ijkl}$ be its curvature tensor. For
any constants $h_{ij}^\mu$ satisfying (\ref{2.2}), there exist
constants $\alpha_{ijk}^l$ such that the map
$u:\mathbb{R}^{n}\rightarrow\mathbb{R}^{n(n+1)/2}$ in (\ref{2.1a})
satisfies
\begin{equation*}
du\cdot du-g=O(|x|^{3})\quad\text{as }|x|\rightarrow 0.
\end{equation*}
\end{lemma}

\begin{proof}  In the following, we denote derivatives of components of $u$ evaluated at the origin by
$u_{i}^{k}=\partial_{i}u^{k}(0)$,
$u_{ij}^{k}=\partial_{ij}u^{k}(0)$, etc. All quantities in the proof are evaluated
at the origin.
We need to find $\alpha_{ijk}^l$ so that (\ref{2.4}) holds. We now write (\ref{2.4}) in the form
\begin{equation}\label{2.4a}
\partial_{kl}g_{ij}=u_{ik}\cdot u_{lj}+u_{il}\cdot u_{jk}
+u_{j}\cdot u_{ikl}+u_{i}\cdot u_{ljk},
\end{equation}
and treat (\ref{2.4a}) as a linear system for $\alpha_{ijk}^l$. A
simple calculation yields that the total number of equations $A$
and unknowns $B$ are given by
\begin{equation*}
A=\left(\frac{n(n+1)}{2}\right)^{2},\quad
B=n\sum_{i=1}^{n}\frac{i(i+1)}{2}=\frac{n^{2}(n+1)(n+2)}{6}.
\end{equation*}
Obviously $A> B$. Hence, (\ref{2.4a}) is an overdetermined system.
Our strategy is to choose a collection of $B$ equations to solve
for $\alpha_{ijk}^l$ and then verify that the rest of the
equations hold automatically under the assumption (\ref{2.2}).

To this end, we first set
$$\tau_{ij}=(i-1)n-\frac12i(i+1)+j\quad\text{for }1\le i<j\le n.$$
Obviously $\tau_{i\, j+1}=\tau_{ij}+1$ and $\tau_{i+1\
i+2}=\tau_{in}+1$. Moreover $\tau_{12}=1$ and $\tau_{n-1\
n}=n(n-1)/2$. Hence $\tau_{ij}$ enumerates the set of integers
$\{1,\cdots, n(n-1)/2\}$ for $1\le i<j\le n$. In fact
$$1=\tau_{12}<\cdots<\tau_{1n}<\tau_{23}\cdots<\tau_{2n}<\cdots
<\tau_{n-1\ n}=\frac12n(n-1).$$ Now, we classify the equation for
$\partial_{kl}g_{ij}$ in (\ref{2.4a}) according to whether the
4-tuple $(i,j,k,l)$ satisfies the conditions
\begin{equation}\label{2.4aa}
i<j,\quad k<l, \quad \tau_{ij}\leq \tau_{kl}.
\end{equation}
We first solve those equations which do not satisfy (\ref{2.4aa}).
To see this, we calculate the number of equations $C$ of this
form:
\begin{equation*}
C=\frac12\cdot \frac{n(n-1)}2\cdot \left(\frac{n(n-1)}2+1\right).
\end{equation*}
%In fact, the equations
%for $g_{ii,kl}$ immediately determine $\alpha^{i}_{ikl}$.
%We then solve the rest of equations %which do not satisfy (\ref{2.4aa}).
The number of degrees of freedom remaining is given by
\begin{equation*}
B-(A-C)=\frac{n(n-1)(n-2)(n-3)}{24}\geq 0.
\end{equation*}
A further calculation shows that this value coincides with the
number of equivalence classes of 4-tuples with all entries
distinct and with $i<j$, $k<l$, $\tau_{ij}< \tau_{kl}$ (here we
say that two tuples are equivalent if they are permutations of
each other). Then in order to use up all the degrees of freedom,
we choose to have one equation of each of these equivalence
classes (where all entries are distinct and $i<j$, $k<l$,
$\tau_{ij}< \tau_{kl}$) satisfied.

The final task is to show that all remaining equations of
(\ref{2.4a}) follow from (\ref{2.2}), which has the form
\begin{equation}\label{2.2a}
R_{ijkl}(0)=u_{ik}\cdot u_{jl}-u_{il}\cdot u_{jk}.
\end{equation}
The remaining equations may be put into three cases.  The first
case occurs when $i=k$, $j=l$, $i<j$. In this case we need to
prove
\begin{equation}\label{2.6}
\partial_{ij}g_{ij}=u_{ii}\cdot u_{jj}+u_{ij}\cdot u_{ij}
+u_{j}\cdot u_{iij}+u_{i}\cdot u_{jji},\quad1\leq i< j\leq n.
\end{equation}
Consider the equations obtained by permuting these indices
\begin{eqnarray*}
\frac{1}{2}\partial_{ii}g_{jj}=&u_{ij}\cdot
u_{ij}+u_{j}\cdot u_{iij},\\
\frac{1}{2}\partial_{jj}g_{ii}=&u_{ij}\cdot
u_{ij}+u_{i}\cdot u_{jji}.
\end{eqnarray*}
These two equations are known to be satisfied.  By a simple addition, we get
\begin{align*}
&\frac{1}{2}\partial_{ii}g_{jj}+\frac{1}{2}\partial_{jj}g_{ii}\\
=&2u_{ij}\cdot
u_{ij}+u_{j}\cdot u_{iij}+u_{i}\cdot u_{jji}\\
=&u_{ij}\cdot
u_{ij}-u_{ii}\cdot u_{jj}+u_{ii}\cdot u_{jj}+u_{ij}\cdot u_{ij}
+u_{j}\cdot u_{iij}+u_{i}\cdot u_{jji}.
\end{align*}
By expressing $R_{ijij}$ in terms of (\ref{2.5}) and (\ref{2.2a}),
we have
\begin{equation*}
\partial_{ij}g_{ij}-\frac{1}{2}\partial_{ii}g_{jj}
-\frac{1}{2}\partial_{jj}g_{ii}
=u_{ii}\cdot u_{jj}-u_{ij}\cdot u_{ij}.
\end{equation*}
A simple comparison yields (\ref{2.6}).

The second case occurs when $k=j$, $i<j$, $j<l$, or $i=k$, $i<j$,
$i<l$, or $j=l$, $i<j$, $k<j$. We first consider $k=j$, $i<j$,
$j<l$. In this case, we need to prove
\begin{equation}\label{2.7}
\partial_{jl}g_{ij}=u_{ij}\cdot u_{jl}+u_{il}\cdot u_{jj}
+u_{j}\cdot u_{ijl}+u_{i}\cdot u_{jjl}.
\end{equation}
Consider the equations obtained by permuting these indices
\begin{eqnarray*}
\partial_{il}g_{jj}&=&2u_{ij}\cdot u_{jl}
+2u_{j}\cdot u_{ijl},\\
\partial_{jj}g_{il}&=&2u_{ij}\cdot u_{jl}
+u_{i}\cdot u_{ljj}+u_{l}\cdot u_{ijj},\\
\partial_{ij}g_{jl}&=&u_{ij}\cdot u_{jl}+u_{jj}\cdot u_{il}
+u_{j}\cdot u_{ijl}+u_{l}\cdot u_{ijj}.
\end{eqnarray*}
All three of these equations are known to be satisfied.  By adding the first two equations
and subtracting the third, we get
\begin{align*}
&\partial_{il}g_{jj}+\partial_{jj}g_{il}-\partial_{ij}g_{jl}\\
=&3u_{ij}\cdot u_{jl}-u_{jj}\cdot u_{il}+u_{j}\cdot u_{ijl}+u_{i}\cdot u_{jjl}\\
=&2u_{ij}\cdot u_{jl}-2u_{ij}\cdot u_{jl}+u_{il}\cdot u_{jj}
+u_{j}\cdot u_{ijl}+u_{i}\cdot u_{jjl}.
\end{align*}
By expressing $R_{ijjl}$ in terms of (\ref{2.5}) and (\ref{2.2a}),
we have
$$\partial_{il}g_{jj}+\partial_{jj}g_{il}
-\partial_{ij}g_{jl}-\partial_{jl}g_{ij} =2u_{ij}\cdot
u_{jl}-2u_{ij}\cdot u_{jl}.$$ A simple comparison yields
(\ref{2.7}), and a similar argument may be used for the cases
$i=k$, $i<j$, $i<l$ and $j=l$, $i<j$, $k<j$.

The third case occurs when $i<j$, $k<l$, $\tau_{ij}< \tau_{kl}$,
and all are distinct. Consider the permutations
\begin{eqnarray*}
\partial_{kl}g_{ij}&=&u_{ik}\cdot u_{lj}+u_{il}\cdot u_{jk}
+u_{j}\cdot u_{ikl}+u_{i}\cdot u_{ljk},\\
\partial_{jl}g_{ik}&=&u_{ij}\cdot u_{kl}+u_{il}\cdot u_{jk}
+u_{i}\cdot u_{jkl}+u_{k}\cdot u_{ijl},\\
\partial_{kj}g_{il}&=&u_{ij}\cdot u_{kl}+u_{ik}\cdot u_{jl}
+u_{i}\cdot u_{ljk}+u_{l}\cdot u_{ijk},\\
\partial_{ij}g_{kl}&=&u_{ik}\cdot u_{lj}+u_{il}\cdot u_{jk}
+u_{k}\cdot u_{lij}+u_{l}\cdot u_{kij},\\
\partial_{ik}g_{jl}&=&u_{ij}\cdot u_{lk}+u_{jk}\cdot u_{il}
+u_{j}\cdot u_{ikl}+u_{l}\cdot u_{ijk},\\
\partial_{il}g_{jk}&=&u_{ij}\cdot u_{kl}+u_{jl}\cdot u_{ik}
+u_{j}\cdot u_{ikl}+u_{k}\cdot u_{ijl}.
\end{eqnarray*}
The last three of these equations are known to be satisfied, where
as the first three need to be established (except for one, which
is known to be satisfied since these three lie in the same
equivalence class of distinct 4-tuples with $i<j$, $k<l$). Using
the last three equations in conjunction with the Gauss equations,
we have
\begin{eqnarray*}
\partial_{kl}g_{ij}-\partial_{jl}g_{ik}&=&-\partial_{ij}g_{kl}
+\partial_{ik}g_{jl}+2R_{kjil},\\
\partial_{kl}g_{ij}-\partial_{jk}g_{il}&=&-\partial_{ij}g_{kl}
+\partial_{il}g_{jk}+2R_{kijl},\\
\partial_{jl}g_{ik}-\partial_{jk}g_{il}&=&-\partial_{ik}g_{jl}
+\partial_{il}g_{jk}+2R_{jikl},
\end{eqnarray*}
and hence
\begin{eqnarray*}
\partial_{kl}g_{ij}-\partial_{jl}g_{ik}&=&u_{ik}\cdot u_{jl}-u_{kl}\cdot u_{ij}+u_{j}\cdot
u_{lik}-u_{k}\cdot u_{lij},\\
\partial_{kl}g_{ij}-\partial_{jk}g_{il}&=&
u_{kj}\cdot u_{il}-u_{kl}\cdot u_{ij}+u_{j}\cdot
u_{kil}-u_{l}\cdot u_{kij},\\
\partial_{jl}g_{ik}-\partial_{jk}g_{il}&=&
u_{jk}\cdot u_{il}-u_{jl}\cdot u_{ik}+u_{k}\cdot
u_{ijl}-u_{l}\cdot u_{ijk}.
\end{eqnarray*}
This may be viewed as three linear equations for the three unknown
$\partial_{kl}g_{ij}$, $\partial_{jl}g_{ik}$, $\partial_{jk}g_{il}$.
Upon solving this system, we find that the solution has the desired
form up to addition of a vector having the form
$(\beta,\beta,\beta)$.  However since at least one of these
equations is known to be satisfied a priori, it follows that
$\beta=0$ so that all are satisfied.
\end{proof}

Our main concern in this paper is the linearized equations
of the isometric embedding. We are interested in such a
linearization only at the formal isometric embedding or its
nearby functions. For this purpose, an approximate isometric
embedding is constructed in Lemma \ref{lemma2.1}. The constants
$h_{ij}^\mu$ are chosen to satisfy (\ref{2.2}). In order to obtain
a simple form of linearized equations, more
assumptions are needed.

\section{Reduction to an $n\times n$
System}\label{section-reduction}

In this section, we reduce the linearization of (\ref{0.2}) to a
first order $n\times n$ system and write the linearized equations
for the isometric embedding system as a perturbation of a first
order differential system with constant coefficients. The linearization
is evaluated at functions which are perturbations of the
approximate isometric embedding in Lemma \ref{lemma2.1}. As we
mentioned in Section \ref{Sec-Intro}, many arguments may be traced
back to \cite{BGY}.

Let $g$ be a metric defined in a neighborhood of the origin in
$\mathbb R^n$. The metric $g$ admits a smooth isometric embedding
into $\mathbb R^{n(n+1)/2}$ if there exists a map $w: \Omega\to
\mathbb{R}^{n(n+1)/2}$ such that
\begin{equation*}
\partial_iw\cdot\partial_j w=g_{ij} \quad\text{for any
}1\le i,j\le n,
\end{equation*}
where $\Omega\subset \mathbb R^n$ is a neighborhood of the origin.
Linearizing at a map
$u:\mathbb{R}^{n}\rightarrow\mathbb{R}^{n(n+1)/2}$ yields the
following linear equation for
$v:\mathbb{R}^{n}\rightarrow\mathbb{R}^{n(n+1)/2}$
\begin{equation}\label{1.1}
\partial_i u\cdot \partial_jv+\partial_j
u\cdot\partial_i v=f_{ij}\quad\text{for any
}1\leq i,j\leq n,
\end{equation}
where $(f_{ij})$ is some smooth symmetric matrix.

In the following, we fix a map $u:
\mathbb{R}^{n}\rightarrow\mathbb{R}^{n(n+1)/2}$ and assume that it
is an embedding. Denote by $\mathcal{M}^{n}$ the corresponding
embedded submanifold. Then $\{\partial_i u(x)\}_{i=1}^{n}$ spans
$T_{x}\mathcal{M}^{n}$ for each $x$. Let
$\{N_{\mu}(x)\}_{\mu=1}^{n(n-1)/2}$ span
$(T_{x}\mathcal{M}^{n})^{\bot}$, the orthogonal complement of
$T_{x}\mathcal{M}^{n}$ in $\mathbb{R}^{n(n+1)/2}$. Denote the
induced metric on $\mathcal{M}^{n}$ by
\begin{equation*}
p_{ij}=\partial_i u\cdot\partial_j u.
\end{equation*}
Then $\partial_{ij}u$ has a decomposition into its tangential and
normal components with respect to $u$ given by
\begin{equation}\label{1.4a}
\partial_{ij}u=\Gamma_{ij}^{k}\partial_k u+H_{ij},
\end{equation}
where $\Gamma_{ij}^{k}$ are the Christoffel symbols corresponding
to $p_{ij}$ and $H_{ij}$ is the second fundamental form. Moreover
we have
\begin{equation}\label{1.5a}
\partial_jN_\mu\cdot\partial_iu=-N_\mu\cdot\partial_{ij}u=
-N_{\mu}\cdot H_{ij}.
\end{equation}
By setting $H_{ij}^{\mu}=H_{ij}\cdot N_{\mu}$, $1\leq\mu\leq
n(n-1)/2$, we have
$$H_{ij}=\sum_{\mu=1}^{n(n-1)/2}H_{ij}^\mu N_{\mu}.$$
We note that (\ref{1.4a}) and (\ref{1.5a}) are simply (\ref{1.4}) and (\ref{1.5}).

In the following, we will express (\ref{1.1}) in another form
which is easier to study. For motivation, we rewrite it as
\begin{equation}\label{1.1aa}
\partial_i(\partial_ju\cdot v)
+\partial_j(\partial_iu\cdot v)-2\partial_{ij}u\cdot
v=f_{ij}\quad\text{for any }1\le i,j\le n.
\end{equation}
Note that $\partial_ju\cdot v$ is a component of the projection of
$v$ into the tangent space $T_x\mathcal M^n$. It is clear from
(\ref{1.1aa}) that the derivatives are only applied to tangential
components of $v$. This suggests that we should decompose $v$
relative to the tangent space and normal space of $\mathcal M^n$.

Set
\begin{equation}\label{1.2}
v=v'+v''=\sum_{k=1}^{n}v^{k}\partial_k
u+\sum_{\mu=1}^{n(n-1)/2} v^{n+\mu}N_{\mu},
\end{equation}
where $v'$ and $v''$ are the tangential and normal components of
$v$ with respect to the embedding $u$. We now derive an equivalent
formulation of (\ref{1.1}) in terms of $v^{k}$ and $v^{n+\mu}$.
Let $\{v_{l}\}_{l=1}^{n}$ be the coordinates of the dual 1-form to
the vector field $v^{l}\partial_l u$, i.e.,
\begin{equation*}
v_{l}=p_{lk}v^{k}\text{ }\text{ and }\text{ }v^{l}=p^{lk}v_{k}.
\end{equation*}
Then
\begin{align*}
\partial_iu\cdot\partial_jv
=&\partial_j(\partial_iu\cdot v)-\partial_{ij}u\cdot v\\
=&\partial_j(p_{il}v^l)-(\Gamma_{ij}^k\partial_ku+H_{ij})
\cdot(v^l\partial_lu+v^{n+\mu}N_\mu)\\
=&\partial_jv_i-\Gamma_{ij}^kv^lp_{kl}-v^{n+\mu}N_\mu\cdot H_{ij}\\
=&\partial_jv_i-\Gamma_{ij}^kv_k-H_{ij}^\mu v^{n+\mu}.
\end{align*}
It follows that (\ref{1.1}) has the form
\begin{equation}\label{1.9}
\partial_j v_{i}+\partial_iv_{j}-2\Gamma_{ij}^{k}v_{k}-2
H_{ij}^\mu v^{n+\mu}=f_{ij}\quad\text{for any }1\le i\le j\le n.
\end{equation}
Moreover this equation may be written invariantly as
\begin{equation*}
\nabla_{i}v_{j}+\nabla_{j}v_{i}-2H_{ij}^\mu v^{n+\mu}=f_{ij},
\end{equation*}
where $\nabla_{i}$ denotes covariant differentiation for 1-forms;
that is, if $\alpha=\alpha_{j}dx^{j}$ then
\begin{equation*}
\nabla_{i}\alpha=(\partial_i
\alpha_{j}-\alpha_{k}\Gamma_{ij}^{k})dx^{j}.
\end{equation*}

Clearly, solving (\ref{1.9}) for $\{v_{i}\}_{i=1}^{n}$ and
$\{v^{n+\mu}\}_{\mu=1}^{n(n-1)/2}$ is equivalent to solving
(\ref{1.1}). This will be accomplished by solving a linear system
of $n(n-1)/2$ algebraic equations for $\{v^{n+\mu}\}$ in terms of
$\{v_{i}\}$, and then inserting this solution into the remaining
$n$ equations to obtain a first order $n\times n$ differential
system in the unknowns $\{v_{i}\}$.

We now specify the algebraic equations used to obtain $\{
v^{n+\mu}\}$. An important observation here is that no derivatives
of $v^{n+\mu}$ are involved in (\ref{1.9}). Consider the
$n(n-1)/2$ equations corresponding to $i<j$ in (\ref{1.9})
\begin{equation}\label{1.10}
H_{ij}^{1} v^{n+1}+\cdots+H_{ij}^{n(n-1)/2}
v^{n+n(n-1)/2}=\phi_{ij} \quad\text{for any } 1\le i<j\le n,
\end{equation}
where
\begin{equation*}
\phi_{ij}=\frac{1}{2}\partial_j v_{i}+ \frac{1}{2}\partial_i
v_{j}-\Gamma_{ij}^{k}v_{k}-\frac{1}{2}f_{ij}\quad\text{for any } 1\le
i<j\leq n.
\end{equation*}
Let
\begin{equation}\label{eq-def$H$}
H(x)=\left(\begin{array}{cccc}
H_{12}^{1} & \cdot & \cdot & H_{12}^{n(n-1)/2} \\
 \cdot &  &  & \cdot \\
 \cdot &  &  & \cdot \\
H_{(n-1)n}^{1} & \cdot & \cdot & H_{(n-1)n}^{n(n-1)/2}
\end{array}\right)
\end{equation}
be the coefficient matrix on the left-hand side of (\ref{1.10}), and
assume that $H$ is invertible with the inverse
\begin{equation*}
H^{-1}=(H^{\mu\tau})\quad\text{for any }1\leq\mu,\tau\leq n(n-1)/2.
\end{equation*}
Note that this assumption of invertibility is not restrictive, since there always exists
a solution of the Gauss equations with this property (see Lemma 3.10 on page 98 of \cite{BCGGG}).
We now solve for $\{v^{n+\mu}\}$ from (\ref{1.10}) in terms of
$\{v_{i}\}$, $\{\nabla v_{i}\}$, and $\{f_{ij}\}$. With
$\tau_{ij}$ defined in the proof of Lemma \ref{lemma2.1}, we have
\begin{equation}\label{1.10a}
v^{n+\mu}=H^{\mu\tau_{ij}}
\big(\frac{1}{2}\partial_j v_{i}+ \frac{1}{2}\partial_i
v_{j}-\Gamma_{ij}^{k}v_{k}-\frac{1}{2}f_{ij}\big).
\end{equation}
We should emphasize that the summation on the right hand side is
taken over $1\le i<j\le n$.

There are $n$ equations in (\ref{1.9}) which are absent in
(\ref{1.10})
\begin{equation}\label{1.11}
\partial_i v_{i}
-\Gamma_{ii}^{k}v_{k}-H_{ii}^\mu v^{n+\mu}
=\frac{1}{2}f_{ii}\quad\text{for any }i=1,\ldots,n.
\end{equation}
By inserting (\ref{1.10a}) into
(\ref{1.11}), we obtain
$$\partial_iv_i-\frac{1}{2}H_{ii}^\mu H^{\mu\tau_{kl}}(\partial_l v_{k}+
\partial_k v_{l})-\big(\Gamma_{ii}^m-H_{ii}^\mu
H^{\mu\tau_{kl}}\Gamma_{kl}^{m}\big)v_m=\frac{1}{2}
\big(f_{ii}-H_{ii}^\mu H^{\mu\tau_{kl}}f_{kl}).$$ We emphasize
that the summation for $k,l$ is only taken for $1\le k<l\le n$,
since $\tau_{kl}$ is defined only for $1\le k<l\le n$. We now
write this as a first order $n\times n$ system for $V=(v_1,\cdots,
v_n)$ of the following form
\begin{equation}\label{1.12}
A^{1}(x)\partial_1V+\cdots+A^{n}(x)
\partial_nV+B(x)V=F(x),
\end{equation}
where $A^k(x)=(A^k_{ij}(x)), B(x)=
(B_{ij}(x))$ and $F(x)$ are given by
\begin{eqnarray*}
A^{k}_{ij}(x)&=&\begin{cases} -H^\mu_{ii}H^{\mu \tau_{jk}}/2&\text{for
}j<k,\\ \delta_{ik} &\text{for }j=k,\\
-H^{\mu}_{ii}H^{\mu\tau_{kj}}/2 &\text{for }j>k,\end{cases}
\\
B_{ij}(x)&=&- \Gamma_{ii}^{j}+\sum_{1\leq k<l\leq
n}H_{ii}^{\mu}H^{\mu
\tau_{kl}}\Gamma_{kl}^{j},\\
F_i(x)&=& \frac{1}{2}\left(f_{ii}-\sum_{1\leq k<l\leq
n}H_{ii}^{\mu}H^{\mu \tau_{kl}}f_{kl}\right).
\end{eqnarray*}
Here $i$ and $j$ denote the rows and columns. Note that the
Christoffel symbols are from the metric induced by $u$, and not
from the given metric $g$. It is now apparent that in order to
solve (\ref{1.1}), it is sufficient to solve (\ref{1.12}) for
$\{v_{i}\}$ and then to find $\{v^{n+\mu}\}$ in terms of
$\{v_{i}\}$, $\{\nabla v_{i}\}$ and $\{f_{ij}\}$ from
(\ref{1.10a}). Therefore the study of the linearization for
(\ref{0.2}) is now reduced to a study of the $n\times n$ system
(\ref{1.12}), as long as the matrix $H$ in (\ref{eq-def$H$}) is
invertible.

We emphasize that, in the calculations so far, $u$ is taken to be
an arbitrary embedding which is not necessarily related to $g$. In
the following, we will choose a special $u$ (the approximate
solution) so that the coefficient matrices of (\ref{1.12}), when
evaluated at $x=0$, are related to the curvature tensor of $g$ at
$x=0$.

To proceed, we let $\{h_{ij}^\mu\}$ be constants satisfying
(\ref{2.2}). Here we emphasize that $R_{ijkl}(0)$ are the
components of the curvature tensor of $g$ at $x=0$.  We then
assume that
$$\text{$u$ {\it is the approximate embedding of} $g$ {\it constructed in Lemma \ref{lemma2.1}}.}$$
By checking our calculations, it is clear that $H$ in
(\ref{eq-def$H$}) satisfies
\begin{equation}\label{eq-def$H$1}
H(0)=\left(\begin{array}{cccc}
h_{12}^{1} & \cdot & \cdot & h_{12}^{n(n-1)/2} \\
 \cdot &  &  & \cdot \\
 \cdot &  &  & \cdot \\
h_{(n-1)n}^{1} & \cdot & \cdot & h_{(n-1)n}^{n(n-1)/2}
\end{array}\right).
\end{equation}
Let $h_{ij}$ be the vector in $\mathbb{R}^{n(n-1)/2}$ defined by
$$h_{ij}=(h_{ij}^1, \cdots, h_{ij}^{n(n-1)/2}).$$
We require that
\begin{equation}\label{2.10}\{h_{ij}\}_{1\le i<j\le n} \text{ forms a basis of }
\mathbb{R}^{n(n-1)/2}.\end{equation}
Then $H(0)$ is invertible and so is $H(x)$ in (\ref{eq-def$H$})
for $x$ sufficiently small. By (\ref{2.10}), we set
\begin{equation}\label{2.10a}
h_{kk}=-2\sum_{1\le i<j\le n}^{n}c_{k}^{ij}
h_{ij}\quad\text{for any }1\leq k\leq
n,
\end{equation}
for some constants $c_{k}^{ij}$.
%Let $(h^{\mu\tau})$ be the inverse
%matrix of $(\omega_{\mu\tau})$. Then
%$$\omega_{\mu\tau_{ij}}h^{\mu\tau_{kl}}=\begin{cases}1&\text{for
%}i<j, k<l\text{ with }i=k, j=l,\\
%0&\text{otherwise}.\end{cases}$$
Therefore
$$A^k_{ij}(0)=\begin{cases} c_{i}^{jk}&\text{for
}j<k,\\ \delta_{ik} &\text{for }j=k,\\
c_{i}^{kj} &\text{for }j>k,\end{cases}$$
and
\begin{eqnarray*}
B_{ij}(0)&=&- \Gamma_{ii}^{j}(0)+\sum_{1\leq k<l\leq
n}c_{i}^{kl}\Gamma_{kl}^{j}(0),\\
F_i(0)&=& \frac{1}{2}\left(f_{ii}(0)-\sum_{1\leq k<l\leq
n}c_{i}^{kl}f_{kl}(0)\right),
\end{eqnarray*}
where $\Gamma_{ij}^k(0)$ are Christoffel symbols of $g$ at $x=0$.
Hence $B_{ij}(0)=0$ by (\ref{2.1}). Lastly, we point out that coefficient
matrices $A^1(0),\ldots,A^n(0)$ in (\ref{1.12}) are related to the curvature
tensor of $g$ at $x=0$. In fact, $\{h_{ij}^\mu\}$ defined by (\ref{2.10}) and
(\ref{2.10a}) satisfies (\ref{2.2}).

To summarize, for $i,j,k=1,\ldots,n$
For $i,j,k=1,\ldots,n$,
let ${\bf c}=\{c_{i}^{kj}\}_{k\neq j}$ be a collection of parameters with
\begin{equation}\label{parameters}c_{i}^{kj}=c_{i}^{jk}\quad\text{for any }
i, j, k \text{ with }j\neq k,
\end{equation}
and set
\begin{align}\label{parameters1}\begin{split}
&c_{i}^{ii}=1\quad\text{for any }i,\\
&c_{i}^{jj}=0\quad\text{for any }i\neq j.
\end{split}
\end{align}
Define $n\times n$ matrices $A^1, \ldots, A^n$ by
$$(A^k)_{ij}=(c_{i}^{kj}).$$ Then we have shown

\begin{lemma}\label{lemma-linearization}
The linear system (\ref{1.12})
at $x=0$ is given by
\begin{equation}\label{1.12z}
A^1\partial_{1}v+
\cdots+
A^n\partial_{n}v=F.
\end{equation}
\end{lemma}

\begin{remark}
The matrices defined above in fact form a basis of the set $F_H$
defined on page 916 in \cite{BGY}, if the second fundamental form $\{h_{ij}^\mu\}$ at $x=0$ is given by (\ref{2.10}) and
(\ref{2.10a}).
\end{remark}

For each $\xi=(\xi_1,\ldots,
\xi_n)\in\mathbb R^n$, define
$$P=P(\xi,{\bf c})=\sum_{i=1}^n\xi_kA^k.$$
This is the {\it principal symbol}, whose components have expressions
\begin{align}\label{eq-principal}
\begin{split}
p_{ii}=&\xi_i+\sum_{k\neq i}c_{i}^{ki}\xi_{k},\\
p_{ij}=&\sum_{k\neq j}c_{i}^{kj}\xi_k\quad\text{for any }i\neq j.
\end{split}
\end{align}
The {\it characteristic variety} is defined by
$$\Sigma({\bf c})=\{\xi\in\mathbb R^n\setminus\{0\}\mid \det P(\xi, {\bf c})=0\}.$$
The next well-known fact asserts that the isometric embedding system is never elliptic
beyond dimension two.

\begin{lemma}\label{lemma-char.variety}
For $n\ge 3$ and any $\bf c$, $\Sigma({\bf c})\neq\emptyset$.
\end{lemma}

This is the second result of Theorem B (v) in \cite{BGY}. In the
present setting, the proof becomes straightforward. Namely, we observe that
$A^{1},\cdots, A^{n}$ are linearly independent by examining the
diagonal elements. Thus the existence of characteristics follows
immediately from \cite{LaxEtAl}.

%The parameters $\bf c$ are related to the curvature tensor in the
%following way. Consider an arbitrary basis
%$\omega=\{\omega_{\tau}\}_{1\le \tau\le n(n-1)/2}$ in
%$\mathbb{R}^{n(n-1)/2}$, and set
%\begin{align*}
%h_{ij}(\omega,\bf{c})&=\omega_{\tau_{ij}}\quad\text{for any }1\le i<j\le n, \\
%h_{kk}(\omega,\bf{c})&=-2
%\sum_{1\le i<j\le n}c^{ij}_k\omega_{\tau_{ij}}\quad\text{for any } 1\leq
%k\leq n.
%\end{align*}
%Then
%$$R_{ijkl}(0)=h_{ik}(\omega,{\bf c})\cdot
%h_{jl}(\omega,{\bf c})-h_{il}(\omega,{\bf c})\cdot
%h_{jk}(\omega,{\bf c}),$$ where $R_{ijkl}(0)$ is the curvature
%tensor of $g$ at $x=0$ and $h_{ik}(\omega,{\bf c})\cdot
%h_{jl}(\omega,\bf{c})$ denotes the dot product in $\mathbb
%R^{n(n-1)/2}$.

%As a remark, we describe how to construct a local isometric
%embedding. There are two steps. In the first step, we find a set
%of parameters ${\bf c}$ so that we are able to solve (\ref{1.12}).
%This step depends on a careful analysis of the characteristic
%variety for (\ref{1.12z}). The second step is to study the
%relationship between the curvature tensor and the second
%fundamental form given by all bases $\{h_{ij}\}_{1\le i< j\le n}$
%in $\mathbb R^{n(n-1)/2}$, and the set of parameters ${\bf c}$.
%From an analytic point of view, the first step is crucial. If we can
%find a specific condition on ${\bf c}$ so that the first step is
%realized, then we may consider all bases $\{h_{ij}\}_{1\le i< j\le
%n}$ in $\mathbb R^{n(n-1)/2}$ to find appropriate conditions on
%the curvature of metrics which admit local isometric embedding.

To end this section, we briefly discuss the principal symbol and characteristic variety for low dimensions. For dimension $n=3$ we
refer the reader to \cite{BGY}. In the case of dimension $n=4$, we will show that the characteristic variety is smooth under generic
conditions on small parameters. To this end, consider the condition
\begin{equation}\label{condition1}
c_{j}^{ik}c_{i}^{jl}\neq c_{j}^{il}c_{i}^{jk}\text{ }\text{ }\text{ for all }\text{ }\text{ }i\neq j,
\end{equation}
where $k$ and $l$ are the remaining two elements of the set $\{1,2,3,4\}\setminus\{i,j\}$, and
consider the four inequalities
\begin{align}\label{condition2}
\begin{split}
c_{2}^{14}c_{4}^{13}c_{3}^{12}\neq& c_{4}^{12}c_{3}^{14}c_{2}^{13},\\
c_{4}^{21}c_{3}^{24}c_{1}^{23}\neq& c_{1}^{24}c_{4}^{23}c_{3}^{21},\\
c_{4}^{31}c_{2}^{34}c_{1}^{32}\neq& c_{1}^{34}c_{4}^{32}c_{2}^{31},\\
c_{2}^{41}c_{1}^{43}c_{3}^{42}\neq& c_{1}^{42}c_{3}^{41}c_{2}^{43}.
\end{split}
\end{align}
These conditions arise naturally when examining the characteristic variety.
In the next result, we write
$t{\bf c}=\{tc_{i}^{kj}\}_{k\neq j}$ for ${\bf c}=\{c_{i}^{kj}\}_{k\neq j}$.

\begin{theorem}\label{theoremdim4}
Let $n=4$. If all elements of ${\bf c}$ satisfy (\ref{condition1})
and (\ref{condition2}), then there exists a constant
$T>0$ (depending on $\bf c$) such that for all $t\in (0,T)$ the
characteristic variety $\Sigma(t{\bf c})$ is smooth.
\end{theorem}

A general result was obtained in \cite{BGY} without the smallness assumption on
the parameters. Here we give an alternative proof.

\begin{proof}
As is discussed in the next section, Bryant, Griffiths, and Yang have shown \cite{BGY}
that the singular part of the characteristic variety consists of points $\xi\in\mathbb{S}^{3}
\subset\mathbb{R}^{4}$ at which all $3\times 3$ determinant minors of the principal symbol vanish. (See Lemma \ref{MainLemma}.)
%Let ${\bf c}(t)=\{t^{1-\delta_{kj}}c_{i}^{kj}\}$ for a small parameter $t>0$.
The principal symbol $P(\x, t{\bf c})=(p_{ij})$ is given by
\begin{align*}
\begin{split}
p_{ii}=&\xi_i+t\sum_{k\neq i}c_{i}^{ki}\xi_{k},\\
p_{ij}=&t\sum_{k\neq j}c_{i}^{kj}\xi_{k},\text{ }\text{ for any }i\neq j.
\end{split}
\end{align*}
Suppose that for all sufficiently small $t$,
singular points $\xi(t)\in\mathbb{S}^{3}$ exist.
In other words, all $3\times 3$ determinant minors of the
principal symbol vanish on $\xi(t)$.
By passing to a subsequence if necessary, we may
assume
$$\xi(t)\rightarrow a=(a_{1},a_{2},a_{3},a_{4})\quad\text{as }t\rightarrow 0.$$ If $P^{i}_{j}$ denotes
the $3\times 3$ minor obtained by deleting the $i$th column and $j$th row, then a simple calculation
shows that
\begin{align*}
\det P^{1}_{1}=& \xi_{2}\xi_{3}\xi_{4}+O(t),\\
\det P^{2}_{2}=& \xi_{1}\xi_{3}\xi_{4}+O(t),\\
\det P^{3}_{3}=& \xi_{1}\xi_{2}\xi_{4}+O(t),\\
\det P^{4}_{4}=& \xi_{1}\xi_{2}\xi_{3}+O(t),
\end{align*}
where we have dropped (and will continue to drop) reference to $t$. Thus at least two components
$a_{i}$ must be zero, say $a_{1}=a_{2}=0$. We may assume that $a_{4}\neq 0$. There are then two cases
to consider, $a_{3}\neq 0$ and $a_{3}=0$.\medskip

\textit{Case 1: $a_{3}\neq 0$.}\medskip

For each $i\neq j$ the components of the principal symbol are given by $p_{ij}=tb_{ij}$, where
\begin{equation*}
b_{ij}=\sum_{k\neq j}c_{i}^{kj}\xi_{k}.
\end{equation*}
Then observe that
\begin{equation*}
\det P^{1}_{2}= t(\xi_{3}\xi_{4}b_{21}+O(t)),\text{ }\text{ }\text{ }\text{ }
\det P^{2}_{1}= t(\xi_{3}\xi_{4}b_{12}+O(t)).
\end{equation*}
It follows that $b_{12}\rightarrow 0$ and $b_{21}\rightarrow 0$ as $t\rightarrow 0$. Thus
\begin{equation*}
c_{2}^{13}a_{3}+c_{2}^{14}a_{4}=0,\text{ }\text{ }\text{ }\text{ }
c_{1}^{23}a_{3}+c_{1}^{24}a_{4}=0,
\end{equation*}
where we have used the symmetry $c_{i}^{kj}=c_{i}^{jk}$. As $a_{3}a_{4}\neq 0$, we must have
\begin{equation}\label{relation1}
c_{2}^{13}c_{1}^{24}=c_{1}^{23}c_{2}^{14}.
\end{equation}

\textit{Case 2: $a_{3}=0$.}\medskip

Observe that
\begin{align*}
\det P^{2}_{3}=& t\xi_{1}\xi_{4}b_{23}+t^{2}(b_{11}b_{23}-b_{31}b_{12})\xi_{4}
+O(t^{2}\max\{|\xi_{1}|,|t|\}),\\
\det P^{3}_{2}=& t\xi_{1}\xi_{4}b_{32}+t^{2}(b_{11}b_{32}-b_{31}b_{12})\xi_{4}
+O(t^{2}\max\{|\xi_{1}|,|t|\}).
\end{align*}
If both of these determinants are zero, then we may multiply the first by $b_{32}$ and the second by $b_{23}$,
and then compare the expressions for $\xi_{1}\xi_{4}b_{23}b_{32}$ to obtain
\begin{equation*}
(b_{21}b_{13}-b_{11}b_{23})b_{32}=(b_{12}b_{31}-b_{11}b_{32})b_{23}+O(t).
\end{equation*}
Recognizing that $\lim_{t\rightarrow 0}b_{ij}=c_{i}^{4j}\xi_{4}$, we find that
\begin{equation}\label{relation2}
c_{2}^{41}c_{1}^{43}c_{3}^{42}= c_{1}^{42}c_{3}^{41}c_{2}^{43}.
\end{equation}

Thus if neither (\ref{relation1}) nor (\ref{relation2}) holds, then there cannot be a limit
point $a=\lim_{t\rightarrow 0}\xi(t)$ of singular points of the characteristic variety with
$a_{1}=a_{2}=0$. By considering all combinations $a_{i}=a_{j}=0$ with $i\neq j$, we obtain the
desired result.
\end{proof}

\section{The Characteristic Variety in Higher Dimensions}\label{section-HighDim}

In this section, we study the characteristic variety of the
linearized isometric embedding system in higher dimensions and
prove Theorem \ref{theorem-nonempty}. As in Section \ref{section-reduction},
let ${\bf c}=\{c_{i}^{kj}\}_{k\neq j}$ be a collection of parameters satisfying (\ref{parameters}).
For any $\xi=(\xi_1,\cdots,\xi_n)\in\mathbb R^n$, we define an $n\times n$ matrix
$P=P(\xi,{\bf c})=(p_{ij})$ by
\begin{align}\label{eq-CharVar}
\begin{split}
p_{ii}=&\xi_i+\sum_{k\neq i}c_{i}^{ki}\xi_{k},\\
p_{ij}=&\sum_{k\neq j}c_{i}^{kj}\xi_k\quad\text{for any }i\neq j.
\end{split}
\end{align}
The matrix $P$ is the principal symbol associated with the
equivalent linearized isometric embedding system. Then the
characteristic variety $\Sigma=\Sigma({\bf c})$ is given by
$$\Sigma({\bf c})=\{\xi\in\mathbb R^n\setminus\{0\}\mid \det P(\xi, {\bf c})=0\}.$$
We observe that $\Sigma({\bf c})$ can be defined alternatively as
$$\Sigma({\bf c})=\{\xi\in\mathbb R^n\setminus\{0\}\mid \text{ the rank of } P(\xi, {\bf c})\le n-1\}.$$
Next define
$$\Sigma_{\text{sing}}({\bf c})=\{\xi\in\mathbb R^n\setminus\{0\}\mid \text{ the rank of } P(\xi, {\bf c})\le n-2\}.$$
We recall the following result.

\begin{lemma}\label{MainLemma}$ \Sigma_{\text{sing}}({\bf c})$ is the
singular part of the characteristic variety $\Sigma({\bf c})$.\end{lemma}

Lemma \ref{MainLemma} is proved in \cite{BGY}. (See \cite{BGY}, Theorem B.) It is
also proved
(\cite{BGY}, Corollary 1.c.6) that $\Sigma_{\text{sing}}({\bf c})$ is not
empty for $n=6, 10, 14, \ldots$ Furthermore it was conjectured (\cite{BGY}, page 920) that
$\Sigma_{\text{sing}}({\bf c})$ is not empty for any $n\ge 5$. The goal of this section is
to prove that $\Sigma_{\text{sing}}({\bf c})$ is not empty for
sufficiently small ${\bf c}$ if $n\ge 5$, and to estimate its size.

By Lemma \ref{MainLemma} $\Sigma_{\text{sing}}({\bf c})$ consists of those
points $\xi$ where all $(n-1)\times(n-1)$ minors of $P(\xi, {\bf c})$ have zero
determinant. Although there seem to be many algebraic equations
involved with this statement, as we will see, there are in fact
only {\it four} under appropriate conditions. The main tool used to reduce
the number of equations is the following result from linear algebra.

\begin{lemma}\label{Lemma-LinearDepend}
Let $v_1,\ldots, v_n$ be $n$ vectors in a vector space. Assume
that $v_1,\ldots$, $v_{n-1}$ and $v_1,\ldots, v_{n-2}$,$v_n$ are
each linearly dependent and that $v_1,\ldots, v_{n-2}$ are
linearly independent. Then any subset of $n-1$ vectors from
$\{v_1,\ldots, v_n\}$ is linearly dependent.
\end{lemma}

We note that $v_1,\ldots, v_{n-2}$ are common vectors of the two
sets $\{v_1,\ldots, v_{n-1}\}$ and $\{v_1,\ldots, v_{n-2},v_n\}$.
It is crucial to assume that $v_1,\ldots, v_{n-2}$ are linearly
independent. We may consider $e_1, \ldots, e_{n-1}, 0$ in $\mathbb
R^{n-1}$, where $e_1, \ldots, e_{n-1}$ form a basis in $\mathbb
R^{n-1}$. Obviously, $e_1, \ldots, e_{n-1}$ are linearly
independent. However, replacing any $e_i$ by the zero vector
yields a linearly dependent set.

\begin{proof}
The proof is a simple argument from linear algebra. Consider
$v_2,\ldots, v_n$. If either $\{v_2,\ldots,v_{n-1}\}$ or
$\{v_2,\ldots, v_{n-2},v_n\}$ is linearly dependent, so is
$\{v_2,\ldots, v_n\}$. We assume that both
$\{v_2,\ldots,v_{n-1}\}$ and $\{v_2,\ldots, v_{n-2},v_n\}$ are
linearly independent. Then by the linear dependence of
$\{v_1,\ldots, v_{n-1}\}$ and $\{v_1,\ldots, v_{n-2}, v_n\}$,
there exist constants $c_2,\ldots, c_{n-1}$ and $d_2,\ldots,
d_{n-2},d_n$ such that
\begin{align*}
v_1=&c_2v_2+\cdots+c_{n-2}v_{n-2}+c_{n-1}v_{n-1},\\
v_1=&d_2v_2+\cdots+d_{n-2}v_{n-2}+d_nv_n.
\end{align*}
Note that $c_{n-1}\neq 0$ and $d_n\neq 0$, otherwise
$\{v_1,\ldots, v_{n-2}\}$ is linearly dependent, which contradicts
the assumption. By taking a difference, we have
$$(c_2-d_2)v_2+\cdots+(c_{n-2}-d_{n-2})v_{n-2}+c_{n-1}v_{n-1}-d_{n}v_{n}=0.$$
This is a nontrivial combination since $c_{n-1}d_n\neq 0$.
\end{proof}

As an application, we discuss conditions under which all minors of
a matrix have zero determinant. We will present only a
simple case. For a matrix $P$, let $P^i_j$ be the minor obtained
by deleting the $i$-th row and $j$-th column from $P$.

\begin{lemma}\label{Lemma-ZeroMinors}
Let $P$ be an $n\times n$ matrix with $n\ge 2$, and suppose that the
four minors $P^1_1$, $P^2_2$, $P^1_2$, and $P^2_1$ have zero determinant.
If all $(n-1)\times (n-2)$ and $(n-2)\times (n-1)$ submatrices are of full rank,
then all minors of $P$ have zero determinant.
\end{lemma}

\begin{proof}
First consider minors without the first row of $P$. We already
have $\det P^1_1=\det P^1_2=0$. The common part of
$P^1_1$ and $P^1_2$ is an $(n-1)\times (n-2)$ matrix obtained by
deleting the first row and the first and second column from $P$,
and hence is of full rank.  By Lemma \ref{Lemma-LinearDepend}, all
minors without the first row in $P$ have zero determinant. In
applying Lemma \ref{Lemma-LinearDepend}, we treat the $(n-1)\times n$
submatrix obtained by deleting the first row from $P$ as a
collection of $n$ column vectors. Similarly, all minors without the
second row in $P$ have zero determinant.

Now we consider other minors, say without the third column. From
the discussion above, $P^1_3$ and $P^2_3$ have zero
determinant. The common part of these two minors is
an $(n-2)\times (n-1)$ matrix obtained by deleting the third column
and the first and second rows from $P$, and is of full
rank. Hence, all minors without the third column in $P$ have zero
determinant. Similarly all minors without the first, second, $\cdots$, or
the $n$-th column in $P$ have zero determinant. In conclusion, all
minors have zero determinant.
\end{proof}

It is clear from the proof that the assumption that all
$(n-1)\times (n-2)$ and $(n-2)\times (n-1)$ submatrices are of full
rank can be relaxed. We only need certain submatrices to be of full
rank. However, we point out that certain conditions are indeed
necessary. For example, the $4\times 4$ diagonal matrix
diag$(1,1,1,0)$ has all but one minor with zero determinant. In
our study of the characteristic variety later on, we will not use
Lemma \ref{Lemma-ZeroMinors} directly. Instead, we will examine
whether certain submatrices are of full rank so that we can apply
Lemma \ref{Lemma-LinearDepend}.
%({\bf The last two sentences should read:
%In
%our study of the characteristic variety later on, we will  use
%Lemma \ref{Lemma-ZeroMinors} directly in studying one case.
%In other cases, we will examine
%whether certain submatrices are of full rank so that we can apply
%Lemma \ref{Lemma-LinearDepend}.})

Lemma \ref{Lemma-ZeroMinors} asserts that there are only four
algebraic equations to satisfy, under appropriate conditions, in
order that all minors of the principal symbol have zero
determinant.
%This leaves huge room in higher dimensions.

We will now study the characteristic variety in $\mathbb R^n$. In
order to introduce a smallness assumption on the parameters of
the system (\ref{1.12z}), we again let
$$t{\bf
c}=\{tc_{i}^{kj}\}_{k\neq j},$$ for some small $t$. According to
(\ref{eq-CharVar}), the principal symbol $P(\x, t{\bf c})=(p_{ij})$ is then given by
\begin{align*}
\begin{split}
p_{ii}=&\xi_i+tb_{ii},\\
p_{ij}=&tb_{ij},\text{ }\text{ for any }i\neq j,
\end{split}
\end{align*}
where
\begin{equation*}
b_{ij}=\sum_{k\neq j}c_{i}^{kj}\xi_{k}.
\end{equation*}

We begin with the case $n=5$, where certain calculations are less
formidable. Pick $i,j\in\{1,\ldots,5\}$ with $i\neq j$, and consider
the following generic conditions on the parameters:
\begin{equation}\label{cond1}
c_{i}^{kj}c_{j}^{li}\neq c_{j}^{ki}c_{i}^{lj} \text{ }\text{ for all }\text{
}\text{ }k\neq l, \text{ }\text{ with }\text{ }k\neq i,j \text{
}\text{ and }\text{ }l\neq i,j,
\end{equation}
and
\begin{align}\label{cond2}\begin{split}
c_{p}^{kI}(c_{i}^{lj}c_{j}^{mi}-c_{j}^{li}c_{i}^{mj})
+&c_{p}^{lI}(c_{j}^{ki}c_{i}^{mj}-c_{i}^{kj}c_{j}^{mi})\\
+&c_{p}^{mI}(c_{i}^{kj}c_{j}^{li}-c_{j}^{ki}c_{i}^{lj})\neq 0,
\end{split}\end{align}
for all $p\neq i,j$ where $k,l,m$ are chosen so that
$\{i,j,k,l,m\}=\{1,\ldots,5\}$ and where $I=i$ or $I=j$. In general,
we will say that the parameters satisfy a \textit{generic condition}
if they satisfy a finite number of (homogeneous) polynomial
inequalities.

%$$P(\x,{\bf c})=\left(\begin{matrix}
%\x_1+tb_{11}& tb_{12} & \cdots & tb_{15}\\
%tb_{21} &\x_2+tb_{22}&tb_{23}&tb_{24}&tb_{25}\\
%tb_{31}&c_{123}\x_1+c_{423}\x_4&\x_3&c_{134}\x_1+c_{234}\x_2\\
%c_{214}\x_2+c_{314}\x_3&c_{124}\x_1+c_{324}\x_3&c_{134}\x_1
%+c_{234}\x_2&\x_4\end{matrix}\right).$$

%The next result concerns the set $\Sigma_{\text{sing}}$ in $\mathbb R^4$.

\begin{theorem}\label{Thm-5dim}
Let $n=5$. If all elements $c_{i}^{kj}$ of ${\bf c}$ satisfy the conditions
(\ref{cond1}) and (\ref{cond2}), then for any sufficiently small
$t>0$, $\Sigma_{\text{sing}}(t{\bf c})\cap \mathbb{P}^4$ contains
ten points. Moreover, if the parameters satisfy a further generic
condition, then for any sufficiently small $t>0$,
$\Sigma_{\text{sing}}(t{\bf c})\cap \mathbb{P}^4$ consists of
$10+\alpha+2\beta+3\gamma$ points where $\alpha$, $\beta$, $\gamma$ are nonnegative
integers with $\alpha+\gamma=10$ and $\beta\leq 5$.
\end{theorem}

\begin{proof}
In light of the discussion at the beginning of this section, our
goal will be to construct a curve
$\xi(t)\in\mathbb{S}^{4}\subset\mathbb{R}^{5}$ such that all
$4\times 4$ minor determinants of $P(\xi(t),t{\bf c})$ vanish, for
all $t$ sufficiently small. If this is to occur, then as in the
first paragraph of the proof of Theorem \ref{theoremdim4}, we must
have (after possibly passing to a subsequence)
$$\xi(t)\rightarrow a=(a_{1},\ldots,a_{5}),$$
where two elements of $a$ vanish, say $a_{1}=a_{2}=0$. There are then three
cases to consider, namely: $a_{3}a_{4}a_{5}\neq 0$, $a_{3}=0$ and
$a_{4}a_{5}\neq 0$, $a_{3}=a_{4}=0$ and $a_{5}\neq 0$.\medskip

\textit{Case 1. $a_{1}=a_{2}=0$ and $a_{3}a_{4}a_{5}\neq 0$.}\medskip

Since $a_{1}=a_{2}=0$, we will write
$$\xi_{1}(t)=ty_{1}(t), \quad
\xi_{2}(t)=ty_{2}(t).$$ It follows that
\begin{equation*}
\xi_{1}+tb_{11}=tx_{1}+t^{2}c_{1}^{21}y_{2},\text{ }\text{ }\text{
}\text{ } \xi_{2}+tb_{22}=tx_{2}+t^{2}c_{2}^{12}y_{1},
\end{equation*}
where
\begin{equation*}
x_{i}=y_{i}+\sum_{k>2}c_{i}^{ki}\xi_{k}.
\end{equation*}
Denote by $P_{j}^{i}$  the minor of the principal symbol obtained by
deleting the $i$th row and $j$th column. We will analyze
$$\det P_1^1=\det P_1^2=\det P_2^1=\det P_2^2=0.$$
First,
\begin{equation*}
\det P_{1}^{1}=t\xi_{3}\xi_{4}\xi_{5}x_{2}+O(t^{2}),\text{ }\text{
}\text{ }\text{ }\det
P_{2}^{2}=t\xi_{3}\xi_{4}\xi_{5}x_{1}+O(t^{2}).
\end{equation*}
This implies $x_1\to0$ and $x_2\to0$ as $t\to0$. Hence, we write $$x_{i}(t)=tz_{i}(t)\quad\text{for }i=1,2,$$  for some
$z_{i}$. Moreover we have
\begin{equation*}
\det P_{2}^{1}=t\xi_{3}\xi_{4}\xi_{5}b_{21}+O(t^{2}),\text{ }\text{
}\text{ }\text{ }\det
P_{1}^{2}=t\xi_{3}\xi_{4}\xi_{5}b_{12}+O(t^{2}).
\end{equation*}
Then we have  $b_{12}\to 0$ and $b_{21}\to 0$ as $t\to 0$. This suggests that we write
\begin{equation}\label{algeqn1}
\sum_{k>2}c_{1}^{k2}\xi_{k}=tz_{3}(t),\text{ }\text{ }\text{ }\text{ }
\sum_{k>2}c_{2}^{k1}\xi_{k}=tz_{4}(t),
\end{equation}
for some $z_{3}$ and $z_{4}$, so that
\begin{equation*}
b_{12}=t(c_{1}^{12}y_{1}+z_{3}),\text{ }\text{ }\text{ }\text{ }
b_{21}=t(c_{2}^{21}y_{2}+z_{4}).
\end{equation*}
Upon calculating the four determinants above in terms of the $z_{i}$
we obtain,
\begin{align}\label{dett1}\begin{split}
\det
P_{1}^{1}=t^{2}&
\bigg[\big(z_{2}-c_{2}^{12}\sum_{k>2}c_{1}^{k1}\xi_{k}\big)
\xi_{3}\xi_{4}\xi_{5}
-b_{23}b_{32}\xi_{4}\xi_{5}\\
&\quad\quad-b_{24}b_{42}\xi_{3}\xi_{5}
-b_{25}b_{52}\xi_{3}\xi_{4}\bigg]+O(t^{3}),
\end{split}\end{align}
\begin{align}\label{dett2}\begin{split}
\det
P_{2}^{2}=t^{2}&\bigg[\big(z_{1}-c_{1}^{21}\sum_{k>2}c_{2}^{k2}\xi_{k}\big)
\xi_{3}\xi_{4}\xi_{5}
-b_{13}b_{31}\xi_{4}\xi_{5}\\
&\quad\quad-b_{14}b_{41}\xi_{3}\xi_{5}
-b_{15}b_{51}\xi_{3}\xi_{4}\bigg]+O(t^{3}),
\end{split}\end{align}
\begin{align}\label{dett3}\begin{split}
\det
P_{2}^{1}=t^{2}&\bigg[\big(z_{4}-c_{2}^{12}\sum_{k>2}c_{2}^{k2}\xi_{k}\big)
\xi_{3}\xi_{4}\xi_{5}
-b_{23}b_{31}\xi_{4}\xi_{5}\\
&\quad\quad-b_{24}b_{41}\xi_{3}\xi_{5}
-b_{25}b_{51}\xi_{3}\xi_{4}\bigg]+O(t^{3}),
\end{split}\end{align}
\begin{align}\label{dett4}\begin{split}
\det
P_{1}^{2}=t^{2}&\bigg[\big(z_{3}-c_{1}^{21}\sum_{k>2}c_{1}^{k1}\xi_{k}\big)
\xi_{3}\xi_{4}\xi_{5}
-b_{13}b_{32}\xi_{4}\xi_{5}\\
&\quad\quad-b_{14}b_{42}\xi_{3}\xi_{5}
-b_{15}b_{52}\xi_{3}\xi_{4}\bigg]+O(t^{3}).
\end{split}\end{align}

Define functions $G_{i}$ by
\begin{equation*}
\det P_{2}^{2}=t^{2}G_{1},\text{ }\text{ }\text{ }\det
P_{1}^{1}=t^{2}G_{2},\text{ }\text{ }\text{ }\det P_{1}^{2}
=t^{2}G_{3},\text{ }\text{ }\text{ }\det P_{2}^{1}=t^{2}G_{4}.
\end{equation*}
We would like each $G_{i}$ to be a function of $z_{i}$ and $t$. To
see that this is the case, we recall (\ref{cond1}) with $i=1$,
$j=2$. Since (\ref{cond1}) holds with $k=3$ and $l=4$, we may solve
equations (\ref{algeqn1}) for $\xi_{3}$ and $\xi_{4}$ in terms of
$\xi_{5}$, $z_{3}$, and $z_{4}$. More precisely, for these values of
$k$ and $l$, fix $\xi_{5}(t)=a_{5}$, a nonzero constant. Then solve
to obtain
\begin{align*}
\xi_{3}(t)=a_{3}+&t(c_{1}^{32}c_{2}^{41}-c_{2}^{31}c_{1}^{42})^{-1}
[(c_{2}^{41}z_{3}(t)-c_{1}^{42}z_{4}(t))\\
&+(c_{1}^{42}c_{2}^{51}-c_{2}^{41}c_{1}^{52})a_{5}],
\end{align*}
\begin{align*}
\xi_{4}(t)=a_{4}+&t(c_{1}^{32}c_{2}^{41}-c_{2}^{31}c_{1}^{42})^{-1}
[(c_{1}^{32}z_{4}(t)-c_{2}^{31}z_{3}(t))\\
&+(c_{2}^{31}c_{1}^{52}-c_{1}^{32}c_{2}^{51})a_{5}],
\end{align*}
where $a_{3}$, $a_{4}$, $a_{5}$ satisfy
\begin{equation}\label{algeqna1}
\sum_{k>2}c_{1}^{k2}a_{k}=0,\text{ }\text{ }\text{ }\text{ }\sum_{k>2}
c_{2}^{k1}a_{k}=0.
\end{equation}
Note that if $a_{5}\neq 0$, then $a_{3}a_{4}\neq 0$, in light of
(\ref{cond1}). We now have a map
\begin{equation*}
G=(G_{1},G_{2},G_{3},G_{4}):\mathbb{R}^{4}\times
\mathbb{R}\rightarrow\mathbb{R}^{4},
\end{equation*}
with $G(\overline{z},0)=0$, where $\overline{z}$ may be determined
from (\ref{dett1}), (\ref{dett2}), (\ref{dett3}), and (\ref{dett4}).
Moreover a simple calculation shows that
$DG(\overline{z},0)=a_{3}a_{4}a_{5}I_{4}$, where $I_{4}$ denotes the
$4\times 4$ identity matrix. By the implicit function theorem
there exists $z(t)=(z_{1}(t),z_{2}(t),z_{3}(t),z_{4}(t))$, such that
$G(z(t),t)=0$ for all sufficiently small $t$. Thus we have found a
curve $\xi(t)$ such that
$$\det P_{1}^{1}(t)=\det P_{2}^{2}(t)=\det
P_{2}^{1}(t)=\det P_{1}^{2}(t)=0.$$

We now claim that all remaining minor determinants vanish as well,
on the curve $\xi(t)$. Consider the last three column vectors of the
principal symbol $P(\xi(t),{\bf c}(t))$ minus the first or second
row. These are linearly independent for small $t$ since
$a_{3}a_{4}a_{5}\neq 0$. If we append the first or second column
(again minus the first or second row) then these four vectors are
linearly dependent since $\det P_{1}^{1}(t)=\det P_{2}^{2}(t)=\det
P_{2}^{1}(t)=\det P_{1}^{2}(t)=0$. So by Lemma
\ref{Lemma-LinearDepend}, all minors $P_{j}^{i}$ with $i=1,2$ have zero
determinant.

Now we move on to columns. Consider the first or second column, that
is consider the last three row vectors minus the first or second
column. These are linearly independent, since $a_{3}a_{4}a_{5}\neq
0$. If we append the first or second row (again minus the first or
second column), then these four vectors are linearly dependent since
$\det P_{1}^{1}=\det P_{1}^{2}=0$. Hence by Lemma
\ref{Lemma-LinearDepend}, all minors $P_{j}^{i}$ with $j=1,2$ have zero
determinant.

Now consider columns 3, 4, and 5. Say column 3. The last three rows
(minus the third column) are linearly independent if $b_{31}\neq 0$
or $b_{32}\neq 0$. Moreover, by appending either the first or second
row (minus the third column), we obtain four linearly dependent
vectors by what has been shown above. Thus Lemma
\ref{Lemma-LinearDepend} shows that all minors $P_{j}^{i}$ with $j=3$
have zero determinant. The same conclusion follows for $j=4$ if
$b_{41}\neq 0$ or $b_{42}\neq 0$, and for $j=5$ if $b_{51}\neq 0$ or
$b_{52}\neq 0$. Therefore all minor determinants vanish, if certain
$b_{kl}$ do not vanish at $t=0$. This fact follows directly from
(\ref{cond2}) with $i=1$ and $j=2$. To see this, solve the linear
algebraic equations (\ref{algeqna1}) for $a_{3}$, $a_{4}$, and
$a_{5}$, and then insert the result into $b_{kl}$ to obtain the
desired conclusion.

In summary, we have constructed a curve $\xi(t)\rightarrow
a=(0,0,a_{3},a_{4},a_{5})$, where the (nonzero) components $a_{3}$, $a_{4}$,
$a_{5}$ satisfy the linear algebraic equations (\ref{algeqna1})
(note that there is only one such point $a\in\mathbb{P}^{4}$), such
that all minor determinants of the principal symbol vanish. It
follows that for each sufficiently small $t$, $\xi(t)$ lies in the
singular part of the characteristic variety. By considering all
possible cases of two zero components, that is $a_{i}=a_{j}=0$, we
obtain ten distinct points in $\Sigma_{\text{sing}}(t{\bf c})\cap
\mathbb{P}^4$. This proves the first statement of the theorem.\medskip

\textit{Case 2. $a_{1}=a_{2}=a_{3}=0$ and $a_{4}a_{5}\neq 0$.}\medskip

Following the strategy of Case 1, we will write
$\xi_{i}(t)=ty_{i}(t)$ for $i=1,2,3$. It follows that
\begin{equation*}
\xi_{1}+tb_{11}=tx_{1}+t^{2}(c_{1}^{21}y_{2}+c_{1}^{31}y_{3}),
\end{equation*}
\begin{equation*}
\xi_{2}+tb_{22}=tx_{2}+t^{2}(c_{2}^{12}y_{1}+c_{2}^{32}y_{3}),
\end{equation*}
\begin{equation*}
\xi_{3}+tb_{33}=tx_{3}+t^{2}(c_{3}^{13}y_{1}+c_{3}^{23}y_{2}),
\end{equation*}
where
\begin{equation*}
x_{i}=y_{i}+c_{i}^{4i}\xi_{4}+c_{i}^{5i}\xi_{5}.
\end{equation*}
We will analyze
$$\det P_2^1=\det P_1^2=\det P_3^1=\det P_3^2=0.$$
First,
\begin{align*}
\det
P_{2}^{1}&=t^{2}(b_{21}x_{3}-b_{23}b_{31})\xi_{4}\xi_{5}+O(t^{3}),\\
\det
P_{1}^{2}&=t^{2}(b_{12}x_{3}-b_{13}b_{32})\xi_{4}\xi_{5}+O(t^{3}),
\end{align*}
which motivates the following. For functions $z_{i}(t)$, $i=1,2,3,4$, to be determined,
we will write
\begin{equation}\label{eqnforb}
x_{1}=z_{1},\text{ }\text{ }\text{ }x_{2}=z_{2},\text{ }\text{ }\text{ }
x_{3}=\frac{b_{23}(0)b_{31}(0)}{b_{21}(0)}+tz_{3},\text{ }\text{ }\text{ }
b_{12}b_{31}b_{23}-b_{21}b_{13}b_{32}=tz_{4},
\end{equation}
where $b_{ij}(0)$ signifies $b_{ij}$ evaluated at $t=0$. Note that under generic conditions
on the parameters $c_{kij}$ we have that $b_{21}(0)\neq 0$
(that this is possible is a consequence of the way in which $a_4$ and $a_5$ are chosen below).
We claim that knowledge of
$z(t)=(z_{1},\ldots,z_{4})$ is equivalent to knowledge of $\xi(t)$. To see this, fix $a_{5}\neq 0$
and let $a_{4}$ solve the cubic polynomial
\begin{align}\label{eqnfora}
\begin{split}
&(c_{1}^{42}a_{4}+c_{1}^{52}a_{5})(c_{3}^{41}a_{4}+c_{3}^{51}a_{5})(c_{2}^{43}a_{4}+c_{2}^{53}a_{5})\\
=&(c_{2}^{41}a_{4}+c_{2}^{51}a_{5})(c_{1}^{43}a_{4}+c_{1}^{53}a_{5})(c_{3}^{42}a_{4}+c_{3}^{52}a_{5}),
\end{split}
\end{align}
which is the last equation of (\ref{eqnforb}) evaluated at $t=0$. Note that this polynomial
has either one or three nonzero solutions for $a_{4}$ again under generic conditions on the parameters.
We then set $\xi_{5}(t)=a_{5}$ and $\xi_{4}=a_{4}+ty_{4}(t)$. Generic conditions
also imply that $y_{4}$ may be determined in terms of $z_{i}$ and $t$ from (\ref{eqnforb}). Thus our task of
constructing $\xi(t)$ is reduced to finding $z(t)$.

We now calculate four determinants in terms of the $z_{i}$:
\begin{align*}
\det P_{2}^{1}&=t^{3}z_{3}b_{21}\xi_{4}\xi_{5}
+O(t^{3}|z_{1}|,t^{3}|z_{2}|,t^{4}),\\
\det P_{1}^{2}&=t^{3}(b_{21}^{-1}z_{4}+b_{12}z_{3})\xi_{4}\xi_{5}
+O(t^{3}|z_{1}|,t^{3}|z_{2}|,t^{4}),
\end{align*}
and
\begin{align*}
\det P_{3}^{2}&=t^{2}(b_{32}z_{1}-b_{12}b_{31})\xi_{4}\xi_{5}+O(t^{3}),\\
\det P_{3}^{1}&=t^{2}(b_{21}b_{32}-b_{31}z_{2})\xi_{4}\xi_{5}+O(t^{3}).
\end{align*}
We may then define a map
\begin{equation*}
G=(G_{1},G_{2},G_{3},G_{4}):\mathbb{R}^{4}\times\mathbb{R}\rightarrow\mathbb{R}^4,
\end{equation*}
where the components of $G$ are functions of $z$ and $t$, and are given by
\begin{equation*}
\det P_{3}^{2}=t^{2}G_{1},\text{ }\text{ }\text{ }\det P_{3}^{1}=t^{2}G_{2},\text{ }
\text{ }\text{ }\det P_{2}^{1}=t^{3}G_{3},\text{ }\text{ }\text{ }\det P_{1}^{2}=t^{3}G_{4}.
\end{equation*}
Let $\overline{z}$ be such that $G(\overline{z},0)=0$, it then follows that
\begin{equation*}
DG(\overline{z},0)=
a_{4}a_{5}\left(\begin{matrix}
b_{32}(0)&0&0&0\\
0&-b_{31}(0)&0&0\\
\ast&\ast&b_{21}(0)&0\\
\ast&\ast&\ast&b_{21}(0)^{-1}
\end{matrix}\right),
\end{equation*}
where $\ast$ represents an expression which is unimportant. Generic conditions on
the parameters then guarantee that this matrix is invertible. The implicit function
theorem then yields $z(t)$ with $G(z(t),t)=0$ for sufficiently small $t$. Thus we have
constructed a curve $\xi(t)$ on which the four determinants calculated above vanish.

We now show that all other minor determinants of the principal symbol vanish on $\xi(t)$,
under generic conditions on the parameters. Consider columns 1, 4, and 5, minus the first
or second row. These vectors are linearly independent for small $t$, if $b_{21}(0)\neq 0$
and $b_{31}(0)\neq 0$. Since $\det P_{2}^{1}=\det P_{3}^{1}=0$ we may apply Lemma \ref{Lemma-LinearDepend}
to show that all minors constructed from an element in the first two rows have zero determinant.
Now consider the last three rows. If $b_{32}(0)\neq 0$ then these vectors are linearly
independent for small $t$. Thus by using what has been shown above, and applying
Lemma \ref{Lemma-LinearDepend}, we find that all minors constructed from an element in the
first column have zero determinant. The same result holds similarly for the second
and third columns. For columns 4 and 5, we obtain this result if (respectively)
\begin{equation*}
b_{31}(0)b_{42}(0)\neq b_{41}(0)b_{32}(0),\text{ }\text{ }\text{ }
b_{31}(0)b_{52}(0)\neq b_{51}(0)b_{32}(0).
\end{equation*}

In conclusion, we have constructed a curve $\xi(t)\rightarrow
a=(0,0,0,a_{4},a_{5})$, where the (nonzero) components $a_{4}$,
$a_{5}$ satisfy the equation (\ref{eqnfora})
(note that under generic conditions on the parameters there is
either one or three such points $a\in\mathbb{P}^{4}$), such
that all determinant minors of the principal symbol vanish. It
follows that for each sufficiently small $t$, $\xi(t)$ lies in the
singular part of the characteristic variety. By considering all
possible cases of three zero components, that is $a_{i}=a_{j}=a_{k}=0$, we
obtain $\alpha+3\gamma$ distinct points in $\Sigma_{\text{sing}}(t{\bf c})\cap
\mathbb{P}^4$, where $\alpha$ and $\gamma$ are nonnegative integers summing
to ten.\medskip

\textit{Case 3. $a_{1}=a_{2}=a_{3}=a_{4}=0$ and $a_{5}\neq 0$.}\medskip

We proceed in the same way as in the previous two cases. Set $\xi_{i}(t)=ty_{i}(t)$,
$i=1,2,3,4$ and calculate
\begin{align*}
\xi_{1}+tb_{11}&=tz_{1}+t^{2}(c_{1}^{21}y_{2}+c_{1}^{31}y_{3}+c_{1}^{41}y_{4}),\\
\xi_{2}+tb_{22}&=tz_{2}+t^{2}(c_{2}^{12}y_{1}+c_{2}^{32}y_{3}+c_{2}^{42}y_{4}),\\
\xi_{3}+tb_{33}&=tz_{3}+t^{2}(c_{3}^{13}y_{1}+c_{3}^{23}y_{2}+c_{3}^{43}y_{4}),\\
\xi_{4}+tb_{44}&=tz_{4}+t^{2}(c_{4}^{14}y_{1}+c_{4}^{24}y_{2}+c_{4}^{34}y_{4}),
\end{align*}
where
\begin{equation*}
z_{i}=y_{i}+c_{i}^{5i}\xi_{5}.
\end{equation*}
Now calculate four determinants in terms of the $z_{i}$:
\begin{align*}
\det P_{2}^{4}=&t^{3}(z_{1}b_{32}b_{43}+z_{3}b_{12}b_{41}
-z_{1}z_{3}b_{42}\\
&-b_{12}b_{31}b_{43}
+b_{13}b_{31}b_{42}-b_{13}b_{32}b_{41})\xi_{5}+O(t^{4}),
\end{align*}
\begin{align*}
\det P_{2}^{3}=&t^{3}(z_{1}z_{4}b_{32}-z_{1}b_{34}b_{42}-z_{4}b_{12}b_{31}\\
&+b_{12}b_{34}b_{41}+b_{14}b_{31}b_{42}-b_{14}b_{32}b_{41})\xi_{5}+O(t^{4}),
\end{align*}
\begin{align*}
\det P_{1}^{2}=&t^{3}(z_{3}z_{4}b_{12}-z_{4}b_{13}b_{32}-z_{3}b_{14}b_{42}\\
&-b_{12}b_{34}b_{43}
+b_{13}b_{34}b_{42}+b_{14}b_{32}b_{43})\xi_{5}+O(t^{4}),
\end{align*}
\begin{align*}
\det P_{3}^{4}=&t^{3}(z_{1}z_{2}b_{43}-z_{1}b_{23}b_{42}-z_{2}b_{13}b_{41}\\
&-b_{12}b_{21}b_{43}
+b_{12}b_{23}b_{41}+b_{13}b_{21}b_{42})\xi_{5}+O(t^{4}).
\end{align*}
Define a map
\begin{equation*}
G=(G_{1},G_{2},G_{3},G_{4}):\mathbb{R}^{4}\times\mathbb{R}\rightarrow\mathbb{R}^4,
\end{equation*}
where the components of $G$ are functions of $z$ and $t$, and are given by
\begin{equation*}
\det P_{3}^{4}=t^{3}G_{1},\text{ }\text{ }\text{ }\det P_{2}^{4}=t^{3}G_{2},\text{ }
\text{ }\text{ }\det P_{2}^{3}=t^{3}G_{3},\text{ }\text{ }\text{ }\det P_{1}^{2}=t^{3}G_{4}.
\end{equation*}
In order to find $\overline{z}$ such that $G(\overline{z},0)=0$, we proceed as follows. Solve
for $\overline{z}_{3}$ in terms of $\overline{z}_{1}$ from $G_{2}(\overline{z},0)=0$, and solve
for $\overline{z}_{4}$ in terms of $\overline{z}_{1}$ from $G_{3}(\overline{z},0)=0$. Now insert
these expressions for $\overline{z}_{3}$ and $\overline{z}_{4}$ into $G_{4}(\overline{z},0)=0$
to obtain a quadratic polynomial for $\overline{z}_{1}$. Under generic conditions on the parameters,
this polynomial has two nonzero real solutions or no real solutions. In the case of two real solutions
we may then find $\overline{z}_{2}$ from $G_{1}(\overline{z},0)=0$.
Furthermore
\begin{equation*}
DG(\overline{z},0)=
\left(\begin{matrix}
\overline{z}_{2}b_{43}-b_{23}b_{42}&\overline{z}_{1}b_{43}-b_{13}b_{41}&0&0\\
b_{32}b_{43}-\overline{z}_{3}b_{42}&0&b_{12}b_{41}-\overline{z}_{1}b_{42}&0\\
\overline{z}_{4}b_{32}-b_{34}b_{42}&0&0&\overline{z}_{1}b_{32}-b_{12}b_{31}\\
0&0&\overline{z}_{4}b_{12}-b_{14}b_{42}&\overline{z}_{3}b_{12}-b_{13}b_{32}
\end{matrix}\right),
\end{equation*}
which is invertible under further generic conditions. Thus the implicit function theorem gives
a function $z(t)$ with $G(z(t),t)=0$ for small $t$. We have then constructed a curve $\xi(t)$ on
which the four relevant determinant minors vanish.

We now show that all other determinant minors vanish on $\xi(t)$, again assuming appropriate generic
conditions. We will use Lemma \ref{Lemma-LinearDepend} as usual, omitting explicit details. Use that
$\det P_{2}^{3}=\det P_{2}^{4}=0$ to show that all minors constructed from elements of the second row
have zero determinant. With this and $\det P_{1}^{2}=0$ we may then show that all minors constructed
from elements of the second column have zero determinant. Furthermore, all determinant minors constructed
from elements of the third row are zero, by the above and $\det P_{3}^{4}=0$. Now that determinant minors
from two rows all vanish, we may proceed exactly as in Cases 1 and 2 to complete the argument.

To conclude the proof, let us note that in Case 1 we have obtained ten points $a\in\Sigma_{\text{sing}}(t{\bf c})\cap
\mathbb{P}^4$, and in Case 2 we have obtained $\alpha+3\gamma$ points where $\alpha$ and $\gamma$ are nonnegative
integers summing to 10. In Case 3, we imposed generic conditions on the parameters so that the quadratic polynomial
equation satisfied by $\overline{z}_{1}$ has two real solutions or no real solutions. Since there are five different
combinations of four zero components for $a$, we obtain $5\beta$ points in Case 3, where $\beta$ is a nonnegative
integer with $\beta\leq 5$. Thus the total is $10+\alpha+2\beta+3\gamma$.
\end{proof}

The proof of Theorem \ref{Thm-5dim} allows a simple characterization of the limit
set $\Lambda({\bf c})\subset\mathbb {P}^4$ (as $t\rightarrow 0$) of the singular part of the
characteristic variety in dimension 5. Namely, under generic conditions on the parameters,
it consists of the following points:

(i) the ten points $a=[a_{1},\ldots,a_{5}]$ with $a_{i}=a_{j}=0$, $i\neq j$, such that
\begin{equation}\label{eqnA1}
\sum_{k\neq i,j}c_{i}^{kj}a_{k}=0,\text{ }\text{ }\text{ }\text{ }\sum_{k\neq i,j}c_{j}^{ki}a_{k}=0;
\end{equation}

(ii) the $\alpha+3\gamma$ points $a=[a_{1},\ldots,a_{5}]$ with $a_{i}=a_{j}=a_{l}=0$, $i<j<l$,
such that
\begin{align}\label{eqnA2}
\begin{split}
& \bigg(\sum_{k\neq i,j,l}c_{i}^{kj}a_{k}\bigg)
\bigg(\sum_{k\neq i,j,l}c_{j}^{kl}a_{k}\bigg)
\bigg(\sum_{k\neq i,j,l}c_{l}^{ki}a_{k}\bigg)\\
=&\bigg(\sum_{k\neq i,j,l}c_{j}^{ki}a_{k}\bigg)
\bigg(\sum_{k\neq i,j,l}c_{l}^{kj}a_{k}\bigg)
\bigg(\sum_{k\neq i,j,l}c_{i}^{kl}a_{k}\bigg);
\end{split}
\end{align}

(iii) five points
$$[1,0,0,0,0],\ [0,1,0,0,0],\ [0,0,1,0,0],\ [0,0,0,1,0],\ [0,0,0,0,1].$$

Recall that the conditions imposed on the parameters guarantee that equations (\ref{eqnA1})
have exactly one solution, and that equation (\ref{eqnA2}) has either one or three solutions.
We have shown that the singular variety
$\Sigma_{\text{sing}}(t{\bf c})\cap \mathbb{P}^4$
lies in a neighborhood of $\Lambda({\bf c})$ for $t$ sufficiently small.

We now study the singular variety for all higher dimensions $n\ge 5$. Although the following
result is not as precise as in the 5-dimensional result above, we are able to show that the
singular variety contains a significantly large algebraic variety. We believe that a similar
analysis, on a case by case basis, as was carried out in the proof of Theorem \ref{Thm-5dim}
is possible, and will lead to a characterization of the singular variety for small $t$. However
due to the extensive calculations involved, we restrict attention to a single case below. Pick
$i,j\in\{1,\ldots,n\}$ with $i\neq j$, and consider the following generic conditions on the parameters:
\begin{equation}\label{Cond1}
c_{i}^{kj}c_{j}^{li}\neq c_{j}^{ki}c_{i}^{lj} \text{ }\text{ for all }\text{
}\text{ }k\neq l, \text{ }\text{ with }\text{ }k\neq i,j \text{
}\text{ and }\text{ }l\neq i,j,
\end{equation}
and for each $p\neq i,j$ there exists $k,l,m\not\in\{i,j\}$ with $k<l<m$ such that
\begin{align}\label{Cond2}\begin{split}
c_{p}^{kI}(c_{i}^{lj}c_{j}^{mi}-&c_{j}^{li}c_{i}^{mj})
+c_{p}^{lI}(c_{j}^{ki}c_{i}^{mj}-c_{i}^{kj}c_{j}^{mi})\\
&+c_{p}^{mI}(c_{i}^{kj}c_{j}^{li}-c_{j}^{ki}c_{i}^{lj})\neq 0,
\end{split}\end{align}
where $I=i$ or $I=j$.

\begin{theorem}\label{Thm-HighDim}
Let $n\geq 5$. If all elements $c_{i}^{kj}$ of ${\bf c}$ satisfy the conditions
(\ref{Cond1}) and (\ref{Cond2}), then for any sufficiently small
$t>0$, $\Sigma_{\text{sing}}(t{\bf c})\cap \mathbb{P}^{n-1}$ contains
an algebraic variety of dimension $n-5$.
\end{theorem}

\begin{proof}
We will follow a similar strategy as in Case 1 of the proof of
Theorem \ref{Thm-5dim}. Thus, our goal will be to construct a curve
$\xi(t)\in\mathbb{S}^{n-1}\subset\mathbb{R}^{n}$ such that all
$(n-1)\times (n-1)$ determinant minors of $P(\xi(t),t{\bf c})$ vanish, for
all sufficiently small $t$. If this is to occur, then as in the
first paragraph of the proof of Theorem \ref{Thm-5dim}, we must
have (after possibly passing to a subsequence)
$$\xi(t)\rightarrow a=(a_{1},\ldots,a_{n}),$$
where two elements of $a$ vanish, say $a_{i}=a_{j}=0$. Choose the remaining components of $a$
to satisfy the following three properties:
\begin{equation*}
\prod_{k\neq i,j}a_{k}\neq 0,
\end{equation*}
\begin{equation*}
\sum_{k\neq i,j}c_{i}^{kj}a_{k}=0,\text{ }\text{ }\text{ }\text{ }
\sum_{k\neq i,j}c_{j}^{ki}a_{k}=0,
\end{equation*}
and for all $p\neq i,j,$
\begin{equation}\label{EqnB}
\sum_{k\neq i,j}c_{p}^{ki}a_{k}\neq 0\text{ }\text{ }\text{ or }\text{ }\text{ }
\sum_{k\neq i,j}c_{p}^{kj}a_{k}\neq 0.
\end{equation}
The generic conditions (\ref{Cond1}) and (\ref{Cond2}) guarantee that such an
$a$ exists. Moreover, it is clear that the set of all $a\in\mathbb{S}^{n-1}$ satisfying
these properties contains an algebraic variety of dimension $n-5$.

In what follows we will let $i=1$ and $j=2$ for convenience.
As in the first part of the proof of Theorem \ref{Thm-5dim},
we will analyze
$$\det P_1^1=\det P_1^2=\det P_2^1=\det P_2^2=0.$$
Set
$$\xi_{1}(t)=ty_{1}(t), \quad \xi_{2}(t)=ty_{2}(t).$$ Then
\begin{equation*}
\xi_{1}+tb_{11}=tx_{1}+t^{2}c_{1}^{21}y_{2},\text{ }\text{ }\text{
}\text{ } \xi_{2}+tb_{22}=tx_{2}+t^{2}c_{2}^{12}y_{1},
\end{equation*}
where
\begin{equation*}
x_{i}=y_{i}+\sum_{k>2}c_{i}^{ki}\xi_{k}.
\end{equation*}
First,
\begin{equation*}
\det P_{1}^{1}=t\xi_{3}\cdots\xi_{n}x_{2}+O(t^{2}),\text{ }\text{
}\text{ }\text{ }\det
P_{2}^{2}=t\xi_{3}\cdots\xi_{n}x_{1}+O(t^{2}).
\end{equation*}
This motivates us to write
$$x_{i}(t)=tz_{i}(t)\quad\text{for }i=1,2,$$ for some
$z_{i}$. Moreover we have
\begin{equation*}
\det P_{2}^{1}=t\xi_{3}\cdots\xi_{n}b_{21}+O(t^{2}),\text{ }\text{
}\text{ }\text{ }\det
P_{1}^{2}=t\xi_{3}\cdots\xi_{n}b_{12}+O(t^{2}).
\end{equation*}
This implies $b_{12}\to0$ and $b_{21}\to0$ as $t\to0$. This suggests that we write
\begin{equation}\label{Algeqn1}
\sum_{k>2}c_{1}^{k2}\xi_{k}=tz_{3}(t),\text{ }\text{ }\text{ }\text{ }
\sum_{k>2}c_{2}^{k1}\xi_{k}=tz_{4}(t),
\end{equation}
for some $z_{3}$ and $z_{4}$, so that
\begin{equation*}
b_{12}=t(c_{1}^{12}y_{1}+z_{3}),\text{ }\text{ }\text{ }\text{ }
b_{21}=t(c_{2}^{21}y_{2}+z_{4}).
\end{equation*}
Upon calculating the four determinants above in terms of the $z_{i}$
we obtain,
\begin{align}\label{Dett1}\begin{split}
\det
P_{1}^{1}=&t^{2}\bigg[\big(z_{2}-c_{2}^{12}\sum_{k>2}c_{1}^{k1}\xi_{k}\big)
\xi_{3}\cdots\xi_{n}\\
&-b_{23}b_{32}\xi_{4}\cdots\xi_{n}-\cdots
-b_{2n}b_{n2}\xi_{3}\cdots\xi_{n-1}\bigg]+O(t^{3}),
\end{split}\end{align}
\begin{align}\label{Dett2}\begin{split}
\det
P_{2}^{2}=&t^{2}\bigg[\big(z_{1}-c_{1}^{21}\sum_{k>2}c_{2}^{k2}\xi_{k}\big)
\xi_{3}\cdots\xi_{n}\\
&-b_{13}b_{31}\xi_{4}\cdots\xi_{n}-\cdots -b_{1n}b_{n1}\xi_{3}\cdots\xi_{n-1}\bigg]+O(t^{3}),
\end{split}\end{align}
\begin{align}\label{Dett3}\begin{split}
\det
P_{2}^{1}=&t^{2}\bigg[\big(z_{4}-c_{2}^{12}\sum_{k>2}c_{2}^{k2}\xi_{k}\big)
\xi_{3}\cdots\xi_{n}\\
&-b_{23}b_{31}\xi_{4}\cdots\xi_{n}-\cdots
-b_{2n}b_{n1}\xi_{3}\cdots\xi_{n-1}\bigg]+O(t^{3}),
\end{split}\end{align}
\begin{align}\label{Dett4}\begin{split}
\det
P_{1}^{2}=&t^{2}\bigg[\big(z_{3}-c_{1}^{21}\sum_{k>2}c_{1}^{k1}\xi_{k}\big)
\xi_{3}\cdots\xi_{n}\\
&-b_{13}b_{32}\xi_{4}\cdots\xi_{n}-\cdots
-b_{1n}b_{n2}\xi_{3}\cdots\xi_{n-1}\bigg]+O(t^{3}).
\end{split}\end{align}

Define functions $G_{i}$ by
\begin{equation*}
\det P_{2}^{2}=t^{2}G_{1},\text{ }\text{ }\text{ }\det
P_{1}^{1}=t^{2}G_{2},\text{ }\text{ }\text{ }\det P_{1}^{2}
=t^{2}G_{3},\text{ }\text{ }\text{ }\det P_{2}^{1}=t^{2}G_{4}.
\end{equation*}
We would like each $G_{i}$ to be a function of $z_{i}$ and $t$. To
see that this is the case, we recall (\ref{Cond1}) with $i=1$,
$j=2$. Since (\ref{Cond1}) holds with $k=3$ and $l=4$, we may solve
equations (\ref{Algeqn1}) for $\xi_{3}$ and $\xi_{4}$ in terms of
$\xi_{5},\ldots,\xi_{n}$, $z_{3}$, and $z_{4}$. More precisely, for these values of
$k$ and $l$, fix $\xi_{l}(t)=a_{l}$, $l=5,\ldots,n$. Then solve
to obtain
\begin{align*}
\xi_{3}(t)=& a_{3}+t(c_{1}^{32}c_{2}^{41}-c_{2}^{31}c_{1}^{42})^{-1}
[(c_{2}^{41}z_{3}(t)-c_{1}^{42}z_{4}(t))\\
&+(c_{1}^{42}c_{2}^{51}-c_{2}^{41}c_{1}^{52})a_{5}
+\cdots+(c_{1}^{42}c_{2}^{n1}-c_{2}^{41}c_{1}^{n2})a_{n}],
\end{align*}
\begin{align*}
\xi_{4}(t)=& a_{4}+t(c_{1}^{32}c_{2}^{41}-c_{2}^{31}c_{1}^{42})^{-1}
[(c_{1}^{32}z_{4}(t)-c_{2}^{31}z_{3}(t))\\
&+(c_{2}^{31}c_{1}^{52}-c_{1}^{32}c_{2}^{51})a_{5}
+\cdots+(c_{2}^{31}c_{1}^{n2}-c_{1}^{32}c_{2}^{n1})a_{n}].
\end{align*}
We now have a map
\begin{equation*}
G=(G_{1},G_{2},G_{3},G_{4}):\mathbb{R}^{4}\times
\mathbb{R}\rightarrow\mathbb{R}^{4},
\end{equation*}
with $G(\overline{z},0)=0$, where $\overline{z}$ may be determined
from (\ref{Dett1}), (\ref{Dett2}), (\ref{Dett3}) and (\ref{Dett4}).
Moreover a simple calculation shows that
$DG(\overline{z},0)=a_{3}\cdots a_{n}I_{4}$, where $I_{4}$ is the
$4\times 4$ identity matrix. By the implicit function theorem
there exists $z(t)$ such that
$G(z(t),t)=0$ for all sufficiently small $t$. Thus we have found a
curve $\xi(t)$ such that $\det P_{1}^{1}(t)=\det P_{2}^{2}(t)=\det
P_{2}^{1}(t)=\det P_{1}^{2}(t)=0$.

The fact that all remaining determinant minors vanish as well on the curve
$\xi(t)$, follows from the same arguments at the end of Case 1 in the
proof of Theorem \ref{Thm-5dim}. Here one must use that $b_{k1}\neq 0$ or
$b_{k2}\neq 0$ for all $k\neq 1,2$, which follows from (\ref{EqnB}).
\end{proof}

Now we discuss the limit set of $\Sigma_{\text{sing}}(t{\bf
c})\cap \mathbb S^{n-1}$ as $t\to 0$. In the proof of Theorem
\ref{Thm-HighDim} for $n\ge 5$, we constructed an
$(n-5)$-dimensional surface in $\Sigma_{\text{sing}}(t{\bf c})\cap
\mathbb S^{n-1}$ which can be viewed as a perturbation of
\begin{equation*}\label{eq-RemarkHigh1}
\{\x\mid \x_1=\x_2=0\}.
\end{equation*}
This corresponds to Case 1 in the proof of Theorem \ref{Thm-5dim}. Similarly, we can
construct an $(n-5)$-dimensional surface in
$\Sigma_{\text{sing}}(t{\bf c})\cap \mathbb S^{n-1}$
corresponding to Cases 2 and 3 in the proof of Theorem \ref{Thm-5dim}. By making appropriate
permutations, we can write down the limit set $\Lambda({\bf c})$ of $\Sigma_{\text{sing}}(t{\bf
c})\cap \mathbb P^{n-1}$ as $t\to 0$. In fact, under appropriate generic conditions on the
parameters, it consists of all points $a=[a_{1},\ldots,a_{n}]\in\mathbb{P}^{n-1}$ satisfying one
of the following three sets of equations:

(i) $a_{i}=a_{j}=a_{k}=a_{l}=0$ for $i<j<k<l$;

(ii) $a_{i}=a_{j}=a_{l}=0$ for $i<j<l$ such that
\begin{align*} &\bigg(\sum_{k\neq i,j,l}c_{i}^{kj}a_{k}\bigg)
\bigg(\sum_{k\neq i,j,l}c_{j}^{kl}a_{k}\bigg)
\bigg(\sum_{k\neq i,j,l}c_{l}^{ki}a_{k}\bigg)\\
=&\bigg(\sum_{k\neq i,j,l}c_{j}^{ki}a_{k}\bigg)
\bigg(\sum_{k\neq i,j,l}c_{l}^{kj}a_{k}\bigg)
\bigg(\sum_{k\neq i,j,l}c_{i}^{kl}a_{k}\bigg);
\end{align*}

(iii) $a_{i}=a_{j}=0$ for $i<j$ such that
\begin{equation*}\sum_{k\neq i,j}c_{i}^{kj}a_{k}=\sum_{k\neq i,j}c_{j}^{ki}a_{k}=0.
\end{equation*}
In terms of the rescaled principal symbol $\widetilde{P}=(\widetilde{p}_{ij})$ given by
\begin{equation*}
\widetilde{p}_{ij}=\lim_{t\rightarrow 0}t^{-1+\delta_{ij}}p_{ij},
\end{equation*}
we can express $\Lambda({\bf c})$ alternatively by
\begin{align*}
\Lambda({\bf c})=&\bigg(\bigcup_{i<j}\{\widetilde{p}_{ii}=\widetilde{p}_{jj}=
\widetilde{p}_{ij}=\widetilde{p}_{ji}=0\}\bigg)\\
&\bigcup \bigg(\bigcup_{i<j<l}\{\widetilde{p}_{ii}=\widetilde{p}_{jj}=
\widetilde{p}_{ll}
=\widetilde{p}_{ij}\widetilde{p}_{jl}\widetilde{p}_{li}
-\widetilde{p}_{ji}\widetilde{p}_{lj}\widetilde{p}_{il}=0\}
\bigg)\\
&\bigcup \bigg(\bigcup_{i<j<k<l}\{\widetilde{p}_{ii}=\widetilde{p}_{jj}=
\widetilde{p}_{kk}=\widetilde{p}_{ll}=0\}\bigg).
\end{align*}
An important observation here is that $\Lambda({\bf c})\cap \mathbb S^{n-1}$ is
smooth except along intersections of any two subsets in this
expression for $\Lambda({\bf c})$. In other words, these
intersections constitute the singular part of $\Lambda({\bf c})$.\medskip

\begin{remark}
To conclude this paper, we make a remark concerning Theorem 4.4 and Theorem 4.5.
When characteristic varieties are smooth, the corresponding linear differential
systems are of the principal type. Symmetrizers can be constructed and solutions can
be proven to exist. However, when the characteristic varieties are not smooth, a
general existence theory of solutions is not available. It is believed that
symmetrizers can still be constructed if the singular sets of the characteristic
varieties enjoy a simple geometry, as illustrated by these theorems.\end{remark}

\bibliographystyle{amsplain}

\begin{thebibliography}{99}

\bibitem{LaxEtAl} Adams, J., Lax, P., Phillips, R., \emph{On
matrices whose real linear combinations are nonsingular}, Proc.
Amer. Math. Soc., \textbf{16}(1965), 318-322.

\bibitem{BBG} Berger, E., Bryant, R., Griffiths, P., \emph{The Gauss equation
and rigidity of isometric embeddings}, Duke Math. J., \textbf{50}(1983),
803-892.

\bibitem{BCGGG} Bryant, R., Chern, S.-S., Gardner, R., Goldschmidt, H., Griffiths, P., \emph{Exterior
Differential Systems}, Mathematical Sciences Research Institute Publication, Vol. 18, Springer-Verlag,
New York (1991).

\bibitem{BGY} Bryant, R., Griffiths, P., Yang, D.,
\emph{Characteristics and existence of isometric embeddings}, Duke
Math. J., \textbf{50}(1983), 893-994.

\bibitem{Goodman-Yang} Goodman, J., Yang, D., \emph{Local
solvability of nonlinear partial differential equations of real
prinicpal type}, unpublished (1988).

\bibitem {Gunther1989} G\"{u}nther, M., \emph{Zum Einbettungssatz
von J. Nash.},  Math. Nachr., \textbf{144}(1989), 165-187.

\bibitem{Han2005} Han, Q.,
\emph{On the isometric embedding of surfaces with Gauss curvature
changing sign cleanly}, Comm. Pure Appl. Math., \textbf{58}(2005), 285-295.

\bibitem{Han200?}  Han, Q., \emph{Local isometric embedding of surfaces
with Gauss curvature changing sign stably across a curve}, Cal. Var.
\& P.D.E., \textbf{25}(2005), 79-103.

\bibitem{Han-Hong2006} Q. Han \& J.-X. Hong, {\it Isometric Embedding of Riemannian
Manifolds in Euclidean Spaces}, Mathematical Surveys and
Monographs, Volume 130, American Mathematical Society, Providence,
RI, 2006.

\bibitem{HanHongLin2003} Han, Q., Hong, J.-X., Lin, C.-S., \emph{Local
isometric embedding of surfaces with nonpositive Gaussian
curvature}, J. Diff. Geometry, \textbf{63}(2003), 475-520.

\bibitem{HanKhuri} Han, Q., Khuri, M., {\it On the local isometric embedding
in $\mathbb{R}^{3}$ of surfaces with Gaussian curvature of mixed
sign}, Comm. Anal. Geom., \textbf{18}(2010), No. 4, 649-704.

\bibitem{Khuri1} Khuri, M., \emph{The local isometric embedding
in $\mathbb R^3$ of two-demensional Riemannian manifolds with
Gaussian curvature changing sign to finite order on a curve}, J.
Diff. Geometry, \textbf{76}(2007), 249-291.

\bibitem{Khuri2} Khuri, M., \emph{Local solvability of
degenerate Monge-Amp\`{e}re equations and applications to
geometry}, Electron. J. Diff. Eqns., \textbf{2007}(2007), No. 65,
1-37.

\bibitem{Lin1985} Lin, C.-S., \emph{The
local isometric embedding in $\mathbb R^3$ of 2-dimensional
Riemannian manifolds with nonnegative curvature}, J. Diff.
Geometry, \textbf{21}(1985), 213-230.

\bibitem{Lin1986} Lin, C.-S., \emph{The local
isometric embedding in $\mathbb R^3$ of two dimensional Riemannian
manifolds with Gaussian curvature changing sign clearly}, Comm.
Pure Appl. Math., \textbf{39}(1986), 307-326.

\bibitem {Naka-Maeda1989} Nakamura, G., Maeda, Y., \emph{Local
smooth isometric embeddings of low dimensional Riemannian
manifolds into Euclidean spaces}, Tran. Amer. Math. Soc.,
\textbf{313}(1989), 1-51.

\bibitem {Nash1956} Nash, J., \emph{The embedding problem for Riemannian
manifolds}, Ann. of Math., \textbf{63}(1956), 20-63.


\end{thebibliography}

\end{document}